\documentclass[a4paper,12pt]{amsart}

\usepackage{fullpage}
\usepackage{xstring}
\usepackage{amsmath,amssymb,graphicx,psfrag}
\usepackage[T1]{fontenc}
\usepackage{amsmath,amssymb,ifthen}
\usepackage[dvipsnames]{xcolor}
\usepackage{pgfplots}
\usepackage{standalone}
\usepackage{adjustbox}
\usepackage{moreverb}
\usepackage{bm}
\usepackage{longtable}

\DeclareFontFamily{U}{mathx}{\hyphenchar\font45}
\DeclareFontShape{U}{mathx}{m}{n}{
      <5> <6> <7> <8> <9> <10>
      <10.95> <12> <14.4> <17.28> <20.74> <24.88>
      mathx10
      }{}
\DeclareSymbolFont{mathx}{U}{mathx}{m}{n}
\DeclareFontSubstitution{U}{mathx}{m}{n}
\DeclareMathAccent{\widecheck}{0}{mathx}{"71}
\DeclareMathAccent{\wideparen}{0}{mathx}{"75}

\def\N{\mathbb{N}}
\def\R{\mathbb{R}}

\def\kpuno{\kappa_1}
\def\kpdue{\kappa_2}
\def\kptre{\kappa_3}

\def\MM{\mathcal{M}}

\def\PP{\mathcal{P}}
\def\TT{\mathcal{T}}

\def\mesh{\mathcal{T}}

\newcommand{\norm}[3][]{#1\Vert #2#1\Vert_{#3}}
\def\set#1#2{\{#1\,:\,#2\}}

\def\diam{{\rm diam}}

\def\osc{{\rm osc}}

\def\eps{\varepsilon}
\def\rho{\varrho}

\definecolor{lightblue}{rgb}{0.75,0.75,1.00}
\definecolor{lightblue}{rgb}{0.75,0.75,0.75}
\def\lightblue{\color{lightblue}}

\newcommand{\jump}[2][]{J_h#1(#2#1)}

\def\A{{\mathbb A}}

\def\Clin{C_{\rm lin}}%
\def\qlin{q_{\rm lin}}%
\def\Crel{C_{\rm rel}}
\def\Cdrel{C_{\rm drel}}

\def\qed{\hfill$\blacksquare$}

\def\qj{q_2}
\def\Cj{C_2}
\def\reff#1#2{\stackrel{\eqref{#1}}{#2}}

\def\ii{\underline{i}}
\def\qi{q_3}
\def\Ci{C_3}

\def\Cmon{C_{\rm mon}}
\def\Cstab{C_{\rm stab}}

\def\Cbin{C_{\rm bin}}
\def\doubleunderline#1{\underline{\underline{#1}}}%
\def\jjj{\doubleunderline{j}}%
\def\kkk{\doubleunderline{k}}%

\def\V{\mathbb{V}}
\def\P{\mathbb{P}}
\def\dual#1#2{\langle#1\,,\,#2\rangle}
\def\f{\boldsymbol{f}}
\def\u{\boldsymbol{u}}
\def\v{\boldsymbol{v}}
\def\w{\boldsymbol{w}}
\def\U{\boldsymbol{U}}
\def\VV{\boldsymbol{V}}
\def\W{\boldsymbol{W}}
\def\div{\nabla\cdot}

\def\Cdiv{C_{\rm div}}%
\def\close{{\tt close}}
\def\bisect{{\tt bisect}}
\def\binev{{\tt binev}}
\def\refine{{\tt refine}}
\def\Cson{C_{\rm son}}
\def\Cclosure{C_{\rm cls}}

\def\n{\boldsymbol{n}}
\def\jump#1{[\![#1]\!]}
\def\etagal#1[#2]{\eta(\TT_{#1};\boldsymbol{U}_{\TT_{#1}}[P_{#2}],P_{#2})}
\def\etagals#1[#2]{\eta_{#1}^\star}
\def\gal#1[#2]{\boldsymbol{U}_{\TT_{#1}}[P_{#2}]}
\def\gals#1[#2]{\boldsymbol{U}_{#1}^\star}

\def\Cstab{C_{\rm stab}}
\def\Cred{C_{\rm red}}
\def\qred{q_{\rm red}}
\def\Cmark{C_{\rm mark}}
\def\QQ{\mathcal{Q}}
\def\jj{\underline{j}}
\def\kk{\underline{k}}

\def\qk{q_1}
\def\Ck{C_1}
\def\qctr{q_{\rm ctr}}

\def\Cmon{C_{\rm mon}}
\def\Cctr{C_{\rm ctr}}
\def\qctr{q_{\rm ctr}}
\def\Cson{C_{\rm son}}
\def\overlay{\eqref{mesh:overlay}}
\def\sonestimate{\eqref{mesh:sons}}

\newcounter{statement}
\newenvironment{statement}[2][!]{%
\vskip3mm
\hrule
\hrule
\hrule
\vskip1mm
\noindent%
\refstepcounter{statement}%
\bf#2~\thestatement%
\ifthenelse{\equal{#1}{!}}{.\ }{~(#1).\ }%
\it%
}{%
\vskip1mm
\hrule
\hrule
\hrule
\vskip2mm
}

\newenvironment{theorem}[1][!]{\begin{statement}[#1]{Theorem}}{\end{statement}}
\newenvironment{lemma}[1][!]{\begin{statement}[#1]{Lemma}}{\end{statement}}

\newenvironment{remark}[1][!]{\begin{statement}[#1]{Remark}}{\end{statement}}
\newenvironment{algorithm}[1][!]{\begin{statement}[#1]{Algorithm}}{\end{statement}}

\numberwithin{statement}{section}

\makeatletter
\def\@seccntformat#1{\hspace*{4mm}%
  \protect\textup{\protect\@secnumfont
    \ifnum\pdfstrcmp{subsection}{#1}=0 \bfseries\fi
    \csname the#1\endcsname
    \protect\@secnumpunct
  }%
}
\makeatother

\usepackage{fancyhdr}
\advance\footskip0.4cm
\advance\textheight-0.4cm
\calclayout
\newcommand*\patchAmsMathEnvironmentForLineno[1]{%
  \expandafter\let\csname old#1\expandafter\endcsname\csname #1\endcsname
  \expandafter\let\csname oldend#1\expandafter\endcsname\csname end#1\endcsname
  \renewenvironment{#1}%
     {\linenomath\csname old#1\endcsname}%
     {\csname oldend#1\endcsname\endlinenomath}}%
\newcommand*\patchBothAmsMathEnvironmentsForLineno[1]{%
  \patchAmsMathEnvironmentForLineno{#1}%
  \patchAmsMathEnvironmentForLineno{#1*}}%
\AtBeginDocument{%
\patchBothAmsMathEnvironmentsForLineno{equation}%
\patchBothAmsMathEnvironmentsForLineno{align}%
\patchBothAmsMathEnvironmentsForLineno{flalign}%
\patchBothAmsMathEnvironmentsForLineno{alignat}%
\patchBothAmsMathEnvironmentsForLineno{gather}%
\patchBothAmsMathEnvironmentsForLineno{multline}%
}
\usepackage[mathlines]{lineno}

\title{Adaptive Uzawa algorithm for the Stokes equation}

\author{Giovanni Di Fratta}

\address{TU Wien, Institute for Analysis and Scientific Computing, Wiedner Hauptstr.~8--10/E101/4, 1040 Wien, Austria}
\email{\{ giovanni.difratta , dirk.praetorius \} @asc.tuwien.ac.at}
\email{gregor.gantner@asc.tuwien.ac.at\quad\rm(corresponding author)}

\author{Thomas F\"uhrer}

\address{Pontificia Universidad Cat\'olica de Chile, Facultad de Matem\'aticas, Vicku\~{n}a Mackenna 4860, Santiago, Chile}
\email{tofuhrer@mat.uc.cl}

\author{Gregor Gantner}
\author{Dirk Praetorius}

\thanks{\textbf{Acknowledgement.} The authors thankfully acknowledge the support by the Austrian Science Fund (FWF) through grant
P27005 (DP), P29096 (GG, DP), as well as grant F65 (GDF, DP) and by CONICYT through FONDECYT project P11170050 (TF).
Moreover, GG thanks Peter Binev for his explanations on \cite{bd,b15}.
}


\date{\today}

\begin{document}


\begin{abstract}
Based on the Uzawa algorithm, we consider an adaptive finite element method for the Stokes system. We prove linear convergence with optimal algebraic rates for the residual estimator (which is equivalent to the total error), if the arising linear systems are solved iteratively, e.g., by PCG. Our analysis avoids the use of discrete efficiency of the  estimator. Unlike prior work, our adaptive Uzawa algorithm can thus avoid to discretize the given data and does not rely on an interior node property for the refinement.
\end{abstract}

\maketitle

\section{Introduction}

\noindent
The mathematical analysis of adaptive finite element methods (AFEMs) has
significantly increased over the last years. Nowadays, AFEMs are recognized
as a powerful and rigorous tool to efficiently solve partial differential
equations arising in physics and engineering.

\subsection{Model problem}
\label{subsec:model problem}
In this paper, we focus on an adaptive algorithm for the solution of
the steady-state Stokes equations, which after a suitable normalization read
\begin{align} \label{eq:StokesCF}
\begin{split}
 - \Delta \u + \nabla p & =\f 
   \quad \text{in } \Omega, \\
\nabla \cdot \u  &= 0  \quad \text{ in }
  \Omega, \\
 \u  &= 0  \quad  \text{ on }
  \partial \Omega. 
  \end{split}
\end{align}
In the literature, the first equation is referred to as \textit{momentum equation}, the second as  \textit{mass equation}, and the third as   \textit{no-slip boundary condition}.
Here, $\Omega\subset\R^{d}$ with $d \in\{2,3\}$ is a bounded polygonal resp.\ polyhedral Lipschitz domain. Given the body force $\f$, one seeks the velocity field $\u$ of an incompressible fluid  and the associated pressure $p$. With
\begin{align}
 \V := H_0^1(\Omega)^d \, ,\quad  \P := \left\{q\in L^2(\Omega):\int_\Omega q\,dx=0\right\},
\end{align}%
it is well-known that the Stokes problem admits a unique solution
$(\u, p) \in \V  \times \P$, where 
$p$ can be characterized as the unique null average
solution of the elliptic Schur complement equation; see, e.g., \cite{bramble}. 
More precisely, the pressure solves 
the
elliptic equation 
\begin{equation} \label{eq:Schurellipeq}
S p = \nabla \cdot \Delta^{-1} \f
\quad \text{with the Schur complement operator} \quad
S := \nabla \cdot \Delta^{- 1} \nabla: \P \to \P.
\end{equation}
The latter equation can be reformulated as a fixpoint problem for the operator
\begin{equation}\label{eq:N_alpha}
  N_{\alpha} : \P \to \P, \quad q\mapsto (I - \alpha S) q + \alpha \nabla
  \cdot\Delta^{- 1} \f. 
\end{equation}
Note that $S$ is self-adjoint.
Since the norm of self-adjoint operators coincides with their spectral radius and $S$ has positive spectrum, one has that  $\| I-\alpha S \|<1$ whenever $|1-\alpha  \| S \||<1$. 
It follows that $N_{\alpha}$ is a contraction for $0 < \alpha < 2 \, \| S \|^{- 1}$;
see Appendix~\ref{appendix:operator}.
Moreover, elementary calculation proves that $\|S\| \le 1$.
Hence, for all $0 < \alpha < 2$ and any initial guess $p^0 \in \P$, the generalized Richardson iteration 
\begin{equation}
 p^{j+1} :=  N_{\alpha} p^j= (I - \alpha S) p^j + \alpha \nabla \cdot\Delta^{- 1} \f
\end{equation}
converges to the
exact pressure of the Stokes problem. It follows that $\u= \lim_{j
\rightarrow \infty} \u [p^j]$ in $\V $ with
$\u [p^j] :=-\Delta^{- 1} (\f- \nabla p^j)$, so that, at the
continuous level, the full iterative process can be expressed in the
form
\begin{align}\label{eq:UzItScheme}
\begin{split}
    \u [p^j] &=- \Delta^{- 1} (\f- \nabla p^j),\\
    p^{j + 1} &= p^j - \alpha \nabla \cdot \u [p^j].
\end{split}
\end{align}
In the spirit of \cite{ks}, the iterative scheme~\eqref{eq:UzItScheme}, usually
referred to as {\emph{Uzawa algorithm}} for the Stokes problem, is the starting point of our AFEM
analysis.

\subsection{State of the art}
Although AFEMs for the analysis of mixed variational problems issuing
from fluid dynamics have a long history in the engineering and physics
literature, only in the last decade,~\cite{ddu}
introduced an adaptive wavelet method based on the Uzawa algorithm for solving the Stokes problem. 
In~\cite{bmn02}, the adaptive wavelet method is replaced by an AFEM.
Their numerical experiments suggested that the latter algorithm leads to optimal algebraic convergence rates. 
Indeed, by addition of a mesh-coarsening step to this method,~\cite{kon06} proved optimal convergence rates.
Later, in~\cite{ks}, the original algorithm of~\cite{bmn02} was modified by adding an additional loop, which separately controls the triangulations on which  the pressure is discretized. 

We also note that for a standard conforming AFEM with Taylor--Hood elements, the first proofs of  plain convergence
were presented in~\cite{mkv08,sie10}.
The work~\cite{gan14} gives an optimality proof under the assumption that some \textit{general quasi-orthogonality} is satisfied. 
This assumption has only recently been verified in~\cite{fei17}.
For adaptive nonconforming finite element methods,  convergence and optimal  rates have been
investigated and proved in~\cite{bm11,hu13,cpr13}.

\subsection{Adaptive Uzawa FEM} In this work, we further investigate the algorithm of~\cite{ks}, which is described in the following:
Given a  possibly non-conforming partition $\PP_i$ of $\Omega$, we denote by $p_i\in\P_i$ the best approximation to $p$, with respect to  the $S$-induced  energy norm $\norm{\cdot}{\P}$, from the corresponding discrete space $\P_i \subset \P$ of piecewise polynomials of degree $m-1$ with  vanishing integral mean. 
With the corresponding velocity $\u_i:=\u[p_i]$ defined analogously to~\eqref{eq:UzItScheme} and the $L^2$-orthogonal projection $\Pi_i:L^2(\Omega)\to\P_i$, one can show that $(\u_i,p_i)$ is the unique solution of the reduced problem
\begin{align}\label{eq:reduced strong stokes}
\begin{split}
- \Delta \u_i + \nabla p_i & =\f 
   \quad \text{in } \Omega, 
   \\  \Pi_i\div \u_i  &= 0  \quad \text{ in }
  \Omega,\\
\u_i  &= 0  \quad  \text{ on }
  \partial \Omega.
  \end{split}
\end{align}
In general, the velocity $\u_i$ is not discrete, and hence this problem can still not be solved in practice.
It is thus approximated by some FEM approximation $\U_{ijk}\in \V_{ijk}$ (the use of three indices being motivated by the structure of the algorithm based on three different iterations) via a standard adaptive algorithm of the form
\begin{align*}
 \boxed{~\text{SOLVE}~}
 \longrightarrow
 \boxed{~\text{ESTIMATE}~}
 \longrightarrow
 \boxed{~\text{MARK}~}
 \longrightarrow
 \boxed{~\text{REFINE}~}
\end{align*}
 for the vector-valued Poisson problem steered by a weighted-residual error estimator $\eta_{ijk}$. 
Here, $\V_{ijk} \subset \V$ denotes the space of all continuous piecewise polynomials on some conforming triangulation $\TT_{ijk}$, which is a refinement of the possibly non-conforming $\PP_i$.
In the next loop, we apply a discretized version of the Uzawa algorithm~\eqref{eq:UzItScheme} to obtain an approximation $P_{ij}\in\P_i$ of $p_i$. 
Here, the update reads $P_{i(j+1)} = P_{ij}-\Pi_i\div \U_{ijk}$.
The last loop employs an adaptive tree approximation algorithm from~\cite{bd} to obtain a better approximation $p_{i+1}\in\P_{i+1}$ of $p$ on a refinement $\PP_{i+1}$ of the partition $\PP_i$ such that $\vartheta \, \norm{\div \U_{ijk}}{\Omega} \le \norm{\Pi_{i+1} \div \U_{ijk}}{\Omega}$ for some bulk parameter $0<\vartheta<1$.
We will see~in Section~\ref{section:estimation} that $\norm{\div \U_{ijk}}{\Omega}$ is related to $\norm{p-p_i}{\P}$ and 
$\norm{\Pi_{i+1} \div \U_{ijk}}{\Omega}$ to $\norm{p_{i+1}-p_i}{\P}$. 
In contrast to~\cite{ks}, in~\cite{bmn02} the latter loop was not present, since the same triangulation for the discretization of the pressure and the velocity, i.e., $\PP_i=\TT_{ijk}$ was used.

Under the assumption that the right-hand side $\f$ is a piecewise polynomial of degree $m-1$,~\cite{ks} proved that the approximations $\U_{ijk}$ and $P_{ij}$ converge with optimal algebraic rate to the exact solutions $\u$ and $p$. 
To generalize this result for arbitrary $\f$, as in the seminal work~\cite{stevenson07}, which proves optimal convergence of a standard AFEM for the Poisson problem,~\cite{ks} applies an additional loop to resolve the data oscillations appropriately. 
However,~\cite{ks} only outlines the proof of this generalization. 
Moreover, as in the seminal work~\cite{stevenson07}, the analysis of~\cite{ks} hinges on the following interior node property: 
Given marked elements $\MM_{ijk}$ of the current velocity triangulation $\TT_{ijk}$, the next velocity triangulation $\TT_{ij(k+1)}$ is the coarsest refinement via newest vertex bisection (NVB) such that all $T\in\MM_{ijk}$ and all $T'\in\TT_{ijk}$, which share a common $(n-1)$-dimensional hyperface, contain a vertex of $\TT_{ij(k+1)}$ in their interior.
In particular for $n=3$, this property is highly demanding; see, e.g., the 3D refinement pattern in~\cite{egp18}.

\subsection{Contributions of present work} 
In the spirit of~\cite{ckns}, which generalizes~\cite{stevenson07}, we prove that the algorithm of~\cite{ks} without the data approximation loop leads to convergence of the combined error estimator $\eta_{ijk}+\norm{\div\U_{ijk}}{\Omega}$ (which is equivalent to the error plus data oscillations) at optimal algebraic rate with respect to the number of  elements $\#\TT_{ijk}$ if one uses standard newest vertex bisection (without interior node property) for the velocity triangulations. 
We also prove that the combined estimator sequence converges linearly in each step, i.e.,  it essentially contracts uniformly in each step. Moreover, our algorithm allows for the inexact solution of the arising linear systems for the discrete velocities by iterative solvers like PCG. 

On a conceptual level, our proofs show that even for general saddle point problems and adaptive strategies based on Richardson-type iterations, the analysis of rate optimal adaptivity can be conducted without exploiting discrete efficiency estimates of the corresponding {\sl a~posteriori} error estimators.

\subsection{Outline}The paper is organized as follows:
Section~\ref{sec:preliminaries} rewrites the Stokes problem in its variational form, introduces newest vertex bisection, and fixes some notation for the discrete ansatz spaces. 
In Section~\ref{sec:uzawa}, we consider the reduced Stokes problem and the corresponding Galerkin approximations, recall some well-known results on {\sl a~posteriori} error estimation, and introduce the tree approximation Algorithm~\ref{function:refinePressure} from~\cite{bd} as well as our variant of the adaptive Uzawa Algorithm~\ref{function:refinePressure} from~\cite{ks}.
In Section~\ref{sec:convergence}, we state and prove linear convergence of the resulting combined error estimator in each step of the algorithm (Theorem~\ref{theorem:linearconvergence}). 
To this end, we show that each increase of either $i,j,$ or $k$ essentially leads to a uniform contraction of the combined error estimator.
Finally, Section~\ref{sec:rates} is dedicated to the main Theorem~\ref{thm:optimal} on optimal convergence rates for the combined error estimator and its proof.
As an auxiliary result of general interest, Lemma~\ref{lemma1:apx} proves that the two different definitions of approximation classes from the literature, which are either based on the accuracy $\eps > 0$ (see, e.g.,~\cite{stevenson,ks}) or the number of elements $N$ (see, e.g.,~\cite{ckns,axioms}),  are exactly the same.

While all constants in statements of theorems, lemmas, etc.\ are explicitly given, we abbreviate the notation in proofs: For scalar terms $A$ and $B$, we write $A \lesssim B$ to abbreviate $A \le C\,B$, where the generic constant $C>0$ is clear from the context. Moreover, $A \simeq B$ abbreviates $A \lesssim B \lesssim A$.

\section{Preliminaries}\label{sec:preliminaries}

\subsection{Continuous Stokes problem}\label{subsec:contStokesprob}

The vector-valued velocity fields $\v \in \V$ are denoted in boldface, the scalar pressures $q \in \P$ in normal font. Let $\dual{\cdot}{\cdot}_\Omega$ be the $L^2(\Omega)$ scalar product with the corresponding $L^2(\Omega)$ norm $\norm{\cdot}{\Omega}$.
With the bilinear forms $a: \V\times\V \to \R$ and $b: \V\times\P \to \R$ defined by
\begin{align*}
 a(\w,\v) := \dual{\nabla \w}{\nabla \v}_\Omega
 \quad\text{and}\quad
 b(\v,q) := - \dual{\div \v}{q}_\Omega,
\end{align*}
the mixed variational formulation of the Stokes problem \eqref{eq:StokesCF} reads as follows: Given $\f \in L^2(\Omega)^{d}$, let $(\u,p) \in \V\times\P$ be the unique solution to
\begin{align}\label{eq:weakform}
 \begin{array}{rclcll}
 a(\u,\v) &+& b(\v,p) &=& \dual{\f}{\v}_\Omega \quad &\text{for all } \v \in \V,\\
 b(\u,q) && &=& 0 &\text{for all } q \in \P.
 \end{array}
\end{align}

On the velocity space $\V$, we consider the $a(\cdot,\cdot)$-induced energy norm $\norm{\v}{\V} := a(\v,\v)^{1/2} = \norm{\nabla \v}{\Omega} \simeq \norm{\v}{H^1(\Omega)}$. We note that $\div\v \in \P$ for all $\v \in \V$ and
\begin{align}\label{eq:appendix:div}
 \norm{\div \v}{\Omega} \le \norm{\nabla \v}{\Omega} = \norm{\v}{\V}
 \quad \text{for all } \v \in \V,
\end{align}
which follows from integration by parts; see Appendix~\ref{appendix:div}.%

Define the operators $A: \V \to \V^*$, $B: \V \to \P^*$, and $B': \P \to \V^*$ by
\begin{align*}
 A\v := a(\v,\cdot), \quad
 B\v := b(\v,\cdot), \quad
 B'q := b(\cdot,q).
\end{align*}
Then, the Schur complement operator $S := BA^{-1}B' : \P \to \P^*\sim\P$ is bounded, symmetric, and elliptic; see~\cite[Lemma~2.2]{ks}. Thus, it holds that $\| q\|_{\P} := \dual{Sq}{q}^{1/2}_{\Omega} \simeq \| q \|_{\Omega}$ on $\P$. 
More precisely, there exists a constant $\Cdiv\ge1$, which depends only on $\Omega$, such that
	\begin{equation}\label{eq:Uzapape2.2}
	\Cdiv^{-1} \, \| q \|_{\Omega} \le \| q \|_{\P}
	\le \| q \|_{\Omega} \quad \text{for all } q\in\P. 
	\end{equation}
Here, the upper bound with constant $1$ follows from $\| S \| \le 1$, which itself follows from~\eqref{eq:appendix:div}.
\subsection{Partitions, triangulations, and newest vertex bisection (NVB)}
\label{subsec:partitions}

\def\Tnc{\mathbb{T}^{\rm nc}}
\def\Tc{\mathbb{T}^{\rm c}}
Throughout, $\PP$ is a finite (possibly non-conforming) partition of $\Omega$ into compact (non-degenerate) simplices, which is used to discretize $\P$, while $\TT$ is a finite (conforming) triangulation of $\Omega$ into compact (non-degenerate) simplices, which is used to discretize $\V$. Throughout, we use NVB refinement; see, e.g.,~\cite{stevenson,kpp} for the precise mesh-refinement rules.

We write $\PP' := \bisect(\PP,\MM)$ for the partition obtained by \emph{one} bisection of all marked elements $\MM \subseteq \PP$, i.e., $\MM = \PP \backslash \PP'$ and $\#\MM = \#\PP' - \#\PP$. We write $\PP' \in \Tnc(\PP)$, if there exists $J \in \N_0$ and partitions $\PP_{j}$ and $\MM_{j} \subseteq \PP_{j}$ for all $j=0,\dots,J$, such that
\begin{align*}
 \PP = \PP_{0},\quad
 \PP_{j} = \bisect(\PP_{j-1},\MM_{j-1}) \text{ for all } j=1,\dots,J,\quad\text{and}\quad
 \PP' = \PP_{J}.
 \end{align*}

We write $\TT' := \refine(\TT,\MM)$ for the coarsest conforming triangulation such that (at least) all marked elements $\MM \subseteq \TT$ have been bisected, i.e., $\MM \subseteq \TT \backslash \TT'$. We write $\TT' \in \Tc(\TT)$, if there exists $J \in \N_0$ and  triangulations $\TT_{j}$ and $\MM_{j} \subseteq \TT_{j}$ for all $j=0,\dots,J$, such~that
\begin{align*}
\TT = \TT_{0},\quad
\TT_{j} = \refine(\TT_{j-1},\MM_{j-1}) \text{ for all } j=1,\dots,J,\quad\text{and}\quad
\TT' = \TT_{J}.
\end{align*}

Let $\TT_{\rm init}$ be a given initial (conforming) triangulation of $\Omega$. We define the sets
\begin{align}
 \Tnc := \Tnc(\TT_{\rm init})
 \quad \text{and} \quad
 \Tc := \Tc(\TT_{\rm init})
\end{align}
of all non-conforming and conforming NVB refinements of $\TT_{\rm init}$. Clearly, $\Tc \subset \Tnc$. We write $\TT := \close(\PP)$ if $\PP \in \Tnc$ is a partition and $\TT \in \Tc$ is the coarsest (conforming) refinement of $\PP$. Existence and uniqueness of $\TT$ follow from the fact that NVB is a binary refinement rule, and the order of the bisections does not matter. In particular, this also implies that $\refine(\TT,\MM) = \close(\bisect(\TT,\MM))$ for all $\TT \in \Tc$ and $\MM \subseteq \TT$.

It follows from elementary geometric observations that NVB refinement leads only to finitely many shapes of simplices $T$; see, e.g.,~\cite{stevenson}. Hence, all NVB refinements are uniformly $\gamma$-shape regular, i.e.,
\begin{align}
 \gamma := \sup_{\PP \in \Tnc} \max_{T \in \PP} \frac{\diam(T)}{|T|^{1/d}} < \infty.
\end{align}
Finally, we recall the following properties of NVB, where $\Cson, \Cclosure > 0$ are constants, which depend only on $\TT_{\rm init}$ and the space dimension $d \ge 2$:
\begin{enumerate}
\renewcommand{\theenumi}{M\arabic{enumi}}
\bf
\refstepcounter{enumi}
\item[(M1)]\label{mesh:overlay}
\rm 
{\bf overlay estimate:} For all $\PP, \PP' \in \Tnc$, there exists a (unique) coarsest common refinement $\PP \oplus \PP' \in \Tnc(\PP) \cap \Tnc(\PP')$. It holds that $\#(\PP \oplus \PP') \le \#\PP + \#\PP' - \#\TT_{\rm init}$. If $\PP, \PP' \in \Tc$ are conforming, it also holds that $\PP \oplus \PP' \in \Tc$. 
\bf
\refstepcounter{enumi}
\item[(M2)]\label{mesh:sons}
\rm
{\bf finite number of sons:} For all $\TT \in \Tc$, $\MM \subseteq \TT$, and $\TT' :=~\refine(\TT,\MM)$, it holds that $\bigcup \set{T' \in \TT'}{T'\subseteq T}=T$ and  $\#\set{T' \in \TT'}{T'\subseteq T} \le \Cson$ for all $T \in \TT$.
\bf
\refstepcounter{enumi}
\item[(M3)]\label{mesh:closure}
\rm
{\bf mesh-closure estimate:}
For all sequences $\TT_{j} \in \Tc$ such that $\TT_0 = \TT_{\rm init}$ and $\TT_{j} =~\refine(\TT_{j-1}, \MM_{j-1})$ with $\MM_{j-1} \subseteq \TT_{j-1}$  for all ${j} \in \N$, it holds that 
\begin{align}
 \#\TT_{J} - \#\TT_{\rm init} \le \Cclosure \sum_{j=0}^{J - 1} \#\MM_j
 \quad \text{for all } J \in \N_0.
\end{align}
\bf
\refstepcounter{enumi}
\item[(M4)]\label{mesh:closure2}
\rm
{\bf conformity estimate:}
For all partitions $\PP \in \Tnc$, it holds that
\begin{align}
 \#\close(\PP) - \#\TT_{\rm init} \le \Cclosure (\#\PP - \#\TT_{\rm init}).
\end{align}
\end{enumerate}
The overlay estimate~\eqref{mesh:overlay} is first proved in~\cite{stevenson07} for $d = 2$, but the proof transfers to arbitrary dimension $d \ge 2$; see~\cite{ckns}. 
For $d = 2$,~\eqref{mesh:sons} obviously holds with $\Cson = 4$, while it is proved in~\cite{gss14} for general dimension $d \ge 2$.
The closure estimate~\eqref{mesh:closure} is first proved in~\cite{bdd} for $d = 2$. For general $d\ge 2$, it is proved in~\cite{stevenson}. While the proofs of~\cite{bdd,stevenson} require an admissibility condition on $\TT_{\rm init}$, the work~\cite{kpp} proves~\eqref{mesh:closure} for $d = 2$, but arbitrary conforming triangulation $\TT_{\rm init}$. 
It is easy to check that~\eqref{mesh:closure} implies~\eqref{mesh:closure2}; see \cite[Lemma~2.5]{bdd} for a proof in 2D, which, however, directly transfers to arbitrary dimension $d \ge 2$.

\subsection{Discrete function spaces}
\label{subsec:discrete}

Given a fixed polynomial degree $m \in \N$ as well as $\PP \in \Tnc$ and $\TT \in \Tc$, we consider the discrete spaces
\begin{align}
\begin{split}
 \P(\PP) &:= \set{Q_\PP \in \P \;}{\forall \, T\in\PP\quad Q_{\PP}|_T \text{ is polynomial of degree }\le m-1},\\
 \V(\TT) &:= \set{\VV_{\TT} \in \V}{\forall \, T\in\TT\quad \VV_{\TT}|_T \text{ is polynomial of degree }\le m},
\end{split}
\end{align}
which consist of piecewise polynomials.

\begin{remark}
We note that our analysis in principle permits to choose the polynomial degree  for the pressure space $\P(\PP)$ larger than $m-1$.
Indeed, the analysis of \cite{ks}  only exploits the assumption that the degree is not larger than  $m-1$ to prove the local efficiency \cite[Proposition~5.2]{ks}, which we do not require; see also \cite[Remark~3.1]{ks}.
However, since we  investigate (optimal) convergence of error quantities  consisting of pressure as well as velocity terms, enlarging only the degree of the pressure space will in general not affect the best possible convergence rate; see also Remark~\ref{dpr:remark-problem}.
Moreover, both the present paper and \cite{ks} do not allow for degrees smaller than $m-1$, since otherwise the property $\TT\in\Tnc(\PP')\cap \Tc$ could no longer be guaranteed by Algorithm~\ref{function:refinePressure}, 
and this condition is essential to ensure that the pressure meshes of the adaptive Algorithm~\ref{algorithm:uzawa} are coarser than the velocity meshes. 

\end{remark}

\subsection{Auxiliary problems}\label{section:auxiliary}

Let $\PP \in \Tnc$.
Then,
$p_\PP \in \P(\PP)$ denotes the best approximation of the exact pressure $p$ with respect to $\norm{\cdot}{\P}$, i.e.,
\begin{align}
 \norm{p - p_\PP}{\P} = \min_{Q_\PP \in \P(\PP)} \norm{p - Q_\PP}{\P}.
\end{align}
It is well-known that $p_\PP$ can be obtained via the unique solution $(\u_\PP,p_\PP) \in \V \times \P(\PP)$ of the \emph{reduced Stokes problem}
\begin{align}\label{eq:weakform:reduced}
 \begin{array}{rclcll}
 a(\u_\PP,\v) &+& b(\v,p_\PP) &=& \dual{\f}{\v}_\Omega \quad &\text{for all } \v \in \V,\\
 b(\u_\PP,Q_\PP) && &=& 0 &\text{for all } Q_\PP \in \P(\PP);
 \end{array}
\end{align}
see~\cite[Section~4]{ks}.
Note that the second condition can equivalently be stated as $\Pi_{\PP} \div \u_\PP = 0$ in $\Omega$, where $\Pi_{\PP}:L^2(\Omega) \to \P(\PP)$ is the orthogonal projection with respect to $\norm{\cdot}{\Omega}$.
Thus,~\eqref{eq:weakform:reduced} is just the variational formulation of~\eqref{eq:reduced strong stokes}. 

Even though $p_\PP$ is a discrete function, it cannot be computed since $\V$ is infinite dimensional. 
Given $q\in\P$, 
 let $\u[q] \in \V$ be the unique solution to the (vector-valued) Poisson equation
\begin{align}\label{eq:weakform:auxiliary}
 a(\u[q],\v) &= \dual{\f}{\v}_\Omega - b(\v,q) \quad \text{for all } \v \in \V.
\end{align}
Note that 
 $\u_\PP = \u[p_\PP]$.

Finally, let $\TT \in \Tnc(\PP) \cap \Tc$ be a conforming refinement of $\PP$. Then, $\U_\TT[q] \in \V(\TT)$ is the unique solution to  the Galerkin discretization of \eqref{eq:weakform:auxiliary} 
\begin{align}\label{eq:weakform:discrete}
 a(\U_\TT[q],\VV_\TT) = \dual{\f}{\VV_\TT}_\Omega - b(\VV_\TT,q)
 \quad\text{for all } \VV_\TT \in \V(\TT).
\end{align}
Note that $\U_\TT[q]$ is the Galerkin approximation to $\u[q]$ in $\V(\TT)$. Since $\norm{\cdot}{\V}$ denotes the energy norm corresponding to $a(\cdot,\cdot)$, there holds the C\'ea lemma
\begin{align}
 \norm{\u[q]- \U_\TT[q]}{\V} = \min_{\VV_\TT \in \V(\TT)} \norm{\u[q] - \VV_\TT}{\V},
\end{align}
Recall the operators $A,B,B'$ from Section \ref{subsec:contStokesprob}. Note  that   $\u[q]-\u[r]=A^{-1}B'(r-q)$  for arbitrary $q,r\in\P$, which yields that $\norm{\u[q]-\u[r]}{\V}^2=\dual{B'(r-q)}{A^{-1}B'(r-q)}_{\V^*\times\V}$.
By definition of the operator $S = BA^{-1}B'$ and the norm $\norm\cdot{\P}$, we thus see that
\begin{align}\label{eq:u to p}
\norm{\U_\TT[q]-\U_\TT[r]}{\V}\le\norm{\u[q]-\u[r]}{\V}=
\norm{q-r}{\P}.
\end{align}

\subsection{Notational conventions}
Throughout this work, $(\u,p) \in \V \times \P$ denotes the exact solution of the continuous Stokes problem~\eqref{eq:weakform}. All occurring functions $\u_\PP$, $\u[q]$, and $\U_\TT[q]$ are approximations of $\u$. All occurring functions $p_\PP$ and $P_\PP$ are approximations of $p$. We employ bold face symbols for velocity functions, e.g., $\v \in \V$ or $\VV_\TT \in \V(\TT)$, and normal font for pressure functions, e.g., $q \in \P$, $Q_\PP \in \P(\PP)$. Finally, small letters indicate functions, which are continuous or not computable, e.g., $\u$, $p$, and $p_\PP$, while \emph{computable discrete} functions are written with capital letters, e.g., $\U_\TT[Q_\PP]$. The corresponding partitions $\PP \in \Tnc$ resp.\ triangulations $\TT \in \Tc$ are always indicated by indices.
The most important symbols are listed in Appendix~\ref{sec:symbols}.

\subsection{Abbreviate notation for adaptive algorithm}\label{section:adaptive algorithm}

The adaptive algorithm below generates nested partitions $\PP_i \in \Tnc$ and triangulations $\TT_{ijk} \in \Tc$ for certain indices $(i, j, k) \in \QQ \subset \N_0^3$ such that $\TT_{ijk} \in \Tnc(\PP_i)\cap \Tc$. Furthermore, it provides approximations 
\begin{align}
p \approx P_{ij} \in \P_i := \P(\PP_i)
\quad \text{as well as} \quad
\u \approx \U_{ijk} \in \V_{ijk} := \V(\TT_{ijk}).
\end{align}
More precisely and with the notation from Section~\ref{section:auxiliary}, it holds that\footnote{Do not confuse the pressure $p_i$ with the iterates $p^j$ of the exact Uzawa algorithm~\eqref{eq:UzItScheme}. }
\begin{align}
P_{ij}\approx p_i:=p_{\PP_i}
\quad \text{as well as} \quad
\U_{ijk} \approx \U_{\TT_{ijk}}[P_{ij}] \approx \u[P_{ij}] =: \u_{ij}.
\end{align}
Besides this notation, let 
\begin{align}
\Pi_i:=\Pi_{\PP_i} : L^2(\Omega) \to \P(\PP_i)
\end{align}
be the $L^2(\Omega)$-orthogonal projection (with respect to $\norm{\cdot}{\Omega}$) and let
\begin{align}
 \eta_{ijk} := \eta(\TT_{ijk}; \U_{ijk}, P_{ij})\approx \eta(\TT_{ijk}; \gal{ijk}[ij], P_{ij})
\end{align}
be the computable {\sl a~posteriori} error estimator from Section~\ref{section:estimation} below.

\section{Adaptive Uzawa algorithm}\label{sec:uzawa}

\subsection{\textsl{A~posteriori} error estimation}\label{section:estimation}

Throughout this section, let $\PP \in \Tnc$ be a partition of $\Omega\subset\R^{d}$ and $\TT \in \Tnc(\PP) \cap \Tc$ be a conforming refinement.
We recall the residual {\sl a~posteriori} error estimator: 
For $T \in \TT$, $Q_\PP \in \P(\PP)$, and $\VV_\TT \in \V(\TT)$, define
\begin{align}
 \eta_T(\VV_\TT, Q_\PP)^2 := |T|^{2/n} \, \norm{\f - \nabla Q_\PP + \Delta \VV_\TT}{T}^2
 + |T|^{1/n} \, \norm{\jump{Q_\PP\,\n - \nabla \VV_\TT \cdot\n}}{\partial T \cap \Omega}^2,
\end{align}
where $\jump{\cdot}$ denotes the jump of its argument over $\partial T$. 
Then, the error estimator reads
\begin{align}
\eta(\MM; \VV_\TT, Q_\PP)^2 := \sum_{T\in\MM} \eta_T(\VV_\TT, Q_\PP)^2 \quad \text{for all } \MM \subset \TT.
\end{align}
In the following, we recall some important properties of $\eta$ from~\cite{ckns,ks}. We start with the available reliability results.

\begin{lemma}[{\em reliability} {\cite[Prop.~5.1, Prop.~5.5]{ks}}]\label{lem:reliability}
There exists a constant $\Crel >0$ 
such that, for all $Q_\PP \in \P(\PP)$, it holds that
\begin{align}\label{eq:reliability:velocity}
 \norm{\u[Q_\PP]-\U_\TT[Q_\PP]}{\V} 
 \le \Crel \, \eta(\TT; \U_\TT[Q_\PP], Q_\PP).
\end{align}
Moreover, it holds that
\begin{align}\label{eq:reliability:reduced}
 \norm{\u_\PP - \U_\TT[Q_\PP]}{\V} + \norm{p_\PP - Q_\PP}{\P} 
 \le \Crel \, \big( \, \eta(\TT; \U_\TT[Q_\PP], Q_\PP) + \norm{\Pi_\PP \div \U_\TT[Q_\PP]}{\Omega}  \, \big)
\end{align}
as well as 
\begin{align}\label{eq:reliability:stokes}
 \norm{\u - \U_\TT[Q_\PP]}{\V} + \norm{p - Q_\PP}{\P} 
 \le \Crel \, \big( \, \eta(\TT; \U_\TT[Q_\PP], Q_\PP) + \norm{\div \U_\TT[Q_\PP]}{\Omega} \, \big).
\end{align}
The constant $\Crel$  depends only on 
$\gamma$-shape regularity. 
\hfill$\blacksquare$
\end{lemma}

For some fixed discrete pressure $Q_\PP$, we recall the localized upper bound in the current form of~\cite{ckns}, which improves~\cite[Prop.~5.1]{ks}.

\begin{lemma}[{\em discrete reliability} {\cite[Lemma~3.6]{ckns}}]\label{lemma:discrete_reliability}
Let $\widehat{\TT} \in~\Tc(\TT)$. There exists a constant $\Cdrel>0$ 
such that, for all $Q_\PP \in \P(\PP)$, it holds that
\begin{align}\label{eq:discrete reliability}
\norm{\U_{\widehat{\TT}}[Q_\PP] - \U_\TT[Q_\PP]}{\V}
\le \Cdrel \, \eta(\TT\setminus\widehat{\TT}; \U[Q_\PP], Q_\PP).
\end{align}
The constant $\Cdrel$  depends only on 
$\gamma$-shape regularity. 
\hfill$\blacksquare$
\end{lemma}

Next, we note that the 
estimator depends Lipschitz continuously on the arguments. The result is slightly stronger than~\cite[Prop.~5.4]{ks}, but the proof is standard~\cite{ckns}.

\begin{lemma}[{\em stability} {\cite[Prop.~3.3]{ckns}}]\label{lemma:stability}
Let $\widehat{\TT}\in\Tc(\TT)$. 
There exists a constant $\Cstab > 0$ such that, for all $\VV_{\widehat{\TT}} \in \V(\widehat{\TT})$, $\W_{\TT}\in\V(\TT)$, and $Q_\PP, R_\PP \in \P(\PP)$, it holds that
\begin{align}\label{eq:stability}
 | \eta(\TT\cap\widehat{\TT}; \VV_{\widehat{\TT}}, Q_\PP) - \eta(\TT\cap\widehat{\TT}; \W_\TT, R_\PP) |
 \le \Cstab \, \big( \norm{\VV_{\widehat{\TT}} - \W_\TT}{\V} + \norm{Q_\PP - R_\PP}{\P} \big).
\end{align}
The constant $\Cstab$ depends only on the polynomial degree $m$ 
and $\gamma$-shape regularity.\qed
\end{lemma}%

The following reduction property follows from the reduction of the mesh-size on refined elements. The proof is standard~\cite{ckns}. 

\begin{lemma}[{\em reduction} {\cite[Proof of Cor.~3.4]{ckns}}]\label{lemma:reduction}
Let $\widehat{\TT}\in\Tc(\TT)$. 
Let  $Q_\PP \in \P(\PP)$.
Then, with $\qred = 2^{-1/(n+1)}$, there holds the reduction property
\begin{align}\label{eq:reduction}
  \eta(\widehat{\TT}\setminus\TT; 
  \U_{\widehat{\TT}}[Q_\PP], Q_\PP) 
 \le \qred \, \eta(\TT\setminus\widehat{\TT}; 
  \U_{\TT}[Q_\PP], Q_\PP) +\Cred \,  \norm{\U_{\widehat{\TT}}[Q_\PP] - \U_\TT[Q_\PP]}{\V}. 
\end{align}
The constant $\Cred>0$ depends only on the polynomial degree $m$ 
and $\gamma$-shape regularity.\qed
\end{lemma}

Finally, for the divergence contribution to the Stokes error estimator, we recall the following equivalence. The result is slightly stronger than~\cite[Prop.~5.7]{ks}.

\begin{lemma}\label{lemma:estimator:div}
Let $\Cdiv \ge 1$ be the norm equivalence constant from~\eqref{eq:Uzapape2.2}. 
Let $\Pi_\PP : L^2(\Omega) \to \P(\PP)$ be the $L^2(\Omega)$-orthogonal projection.
If $Q_\PP \in \P(\PP)$, then it holds that
\begin{align}\label{eq:estimator:div}
 \norm{\Pi_\PP \div \u[Q_\PP]}{\Omega}
 \le \norm{\div(\u_\PP-\u[Q_\PP])}{\Omega}
 \le \norm{p_\PP - Q_\PP}{\P} 
 \le \Cdiv \, \norm{\Pi_\PP \div \u[Q_\PP]}\Omega. 
\end{align}
If $q\in\P$, it holds that
\begin{align}\label{eq2:estimator:div}
 \norm{\div \u[q]}\Omega 
 \le \norm{p - q}{\P} \le \Cdiv \, \norm{\div \u[q]}\Omega. 
\end{align}%
\end{lemma}%

\begin{proof}
	From the definition of the Schur complement operator, we have that
\begin{equation}\label{eq:fundSDIV}
	\div (\u_\PP-\u[Q_\PP])=S(p_\PP-Q_\PP).
\end{equation}
Taking into account~\eqref{eq:Uzapape2.2}, we obtain that
\begin{align*}
 &\| \div(\u_\PP-\u[Q_\PP]) \|_{\Omega}^2 
~\reff{eq:fundSDIV}= \dual{S (p_\PP-Q_\PP)}{\div(\u_\PP -\u[Q_\PP])}_{\Omega}
 \\&\quad
  = \, \dual{p_\PP - Q_\PP}{\div  (\u_\PP-\u[Q_\PP])}_{\P}
  \le  \| p_\PP - Q_\PP\|_{\P} \,
\| {\div} (\u_\PP -\u[Q_\PP])\|_{\P} 
 \\&\quad
\reff{eq:Uzapape2.2} \le  \| p_\PP - Q_\PP \|_{\P} \, \|\div(\u_\PP -\u[Q_\PP])\|_{\Omega}.
\end{align*}
Together with $\Pi_\PP \div\u_\PP = 0$, this proves that 
\begin{equation*}
\norm{\Pi_\PP\div\u[Q_\PP]}{\Omega} 
\le 
\|
{\div}  (\u_\PP
-\u[Q_\PP]) \|_{\Omega} \le \| p_\PP - Q_\PP\|_{\P} .
\end{equation*}
On the other hand, note that 
$\Pi_\PP(p_\PP-Q_\PP) = p_\PP-Q_\PP$. The norm equivalence~\eqref{eq:Uzapape2.2} and the Cauchy-Schwarz inequality thus imply that
\begin{align*}
&\Cdiv \| p_\PP - Q_\PP \|_{\P}\; \| \Pi_\PP \div\u[Q_\PP]\|_{\Omega} 
~\reff{eq:Uzapape2.2} \ge  \| p_\PP - Q_\PP \|_{\Omega} \; \| \Pi_\PP \div\u[Q_\PP]\|_{\Omega} 
\\&\qquad
 \geq  - \dual{p_\PP - Q_\PP}{ \Pi_\PP \div\u[Q_\PP]}_\Omega 
 \,\,=\, \dual{p_\PP - Q_\PP}{ \Pi_\PP \div(\u_\PP-\u[Q_\PP])}_\Omega 
\\&\qquad
 = \dual{p_\PP - Q_\PP}{ \div(\u_\PP-\u[Q_\PP])}_\Omega
\reff{eq:fundSDIV}=  \dual{S(p_\PP - Q_\PP)}{p_\PP - Q_\PP}_\Omega
 =  \| p_\PP - Q_\PP \|^2_{\P}
\end{align*}
and therefore $\| p_{\PP} - Q_\PP \|_{\P}\leq\Cdiv  \| \Pi_{\PP} \div \u[Q_\PP]\|_{\Omega}$.
This concludes the proof of~\eqref{eq:estimator:div}. 
The proof of~\eqref{eq2:estimator:div} follows along the same lines (with $p=p_\PP$ and hence $0 = \div \u = \div \u_\PP$, and $q=Q_{\PP}$).
\end{proof}

\subsection{Adaptive refinement of pressure triangulation}
\label{subsec:binev}

To refine the partitions $\PP_i$, we apply the following algorithm from \cite[Section~2]{b15} (which slightly differs from the well-known \textit{thresholding second algorithm} of~\cite{bd}).
We will use it in Algorithm~\ref{algorithm:uzawa} with parameters $\PP_i,  \TT_{ijk},  \U_{ijk}\approx \u[P_{ij}]$, and $\vartheta$. 
In this context, the idea of Algorithm~\ref{function:refinePressure} is to achieve that  $\norm{p_{i+1}-P_{ij}}{\P}\simeq \norm{\Pi_{i+1}\div u[P_{ij}]}{\Omega}$ dominates $\norm{p-P_{ij}}{\P}\simeq \norm{\div u[P_{ij}]}{\Omega}$ (see Lemma~\ref{lemma:estimator:div}), and to subsequently proceed to the iteration in $j$ and improve the Uzawa approximation.

\begin{algorithm}\label{function:refinePressure}
\textsc{Input:}  Partition  $\PP':=\PP\in \Tnc$, triangulation $\TT\in \Tnc(\PP) \cap \Tc$, function $\VV_\TT \in \V(\TT)$, adaptivity parameter
$0 < \vartheta \le 1$. 

\noindent
\textsc{Loop:} Iterate the following steps \rm (i)--(iii) until
$\vartheta \, \norm{\div \VV_\TT}{\Omega} \le \norm{\Pi_{\PP'} \div \VV_\TT}{\Omega}$:
\begin{itemize}
\item[\rm (i)] Compute
$e(T):=\inf\set{\norm{\div \VV_\TT-Q}{T}^2}{Q\text{ polynomial of degree }m-1}$ for all $T\in\PP'$, for which $e(T)$ has not been already computed.
\item[\rm (ii)] 
For all $T\in\PP'$ for which $\widetilde e(T)$ has not been already defined, define  $\widetilde e (T):=e(T)$ if $T\in\PP$ and $\widetilde e(T):={e(T)\widetilde e(\widetilde T)}/({e(T)+\widetilde e(\widetilde T)})$ otherwise, where $\widetilde T$ denotes the unique father element of $T$.
\item[\rm (iii)]
Choose one element $T\in\PP'$ with $\widetilde e(T)=\max_{T'\in\PP'}\widetilde e(T')$ and employ newest vertex bisection to generate $\PP':=\bisect(\PP',\{T\})$.
\end{itemize}

\noindent
\textsc{Output:}  Triangulation $\PP'=\binev(\PP,\TT,\VV_\TT;\vartheta) \in \Tnc(\PP)$ with $\TT \in \Tnc(\PP') \cap \Tc$.
\end{algorithm}

According to \cite[Theorem~2.1]{b15}, the output $\PP'$ is a quasi-optimal mesh in $\Tnc(\PP)$ with $\vartheta \, \norm{\div \VV_\TT}{\Omega} \le \norm{\Pi_{\PP'} \div \VV_\TT}{\Omega}$:
This means that for all $\vartheta<\vartheta'<1$ and all $\widetilde\PP \in \Tnc(\PP)$ with $\vartheta' \, \norm{\div \VV_\TT}{\Omega} \le \norm{\Pi_{\widetilde\PP} \div \VV_\TT}{\Omega}$, it holds that $\#\PP' - \#\PP \le \Cbin \, (\#\widetilde\PP - \#\PP)$ for some $\Cbin>1$, which depends only on the ratio $(1-\vartheta'^2)/(1-\vartheta^2)$.
The same reference states that the effort of Algorithm~\ref{function:refinePressure} is $\mathcal{O}(\#\TT\log(\#\TT))$ if $0 < \vartheta < 1$.


To obtain optimal algebraic convergence rates of the error estimator, one has to choose $\vartheta$ sufficiently small and $\vartheta'$ sufficiently close to $\vartheta$; see Theorem~\ref{thm:optimal} below.

\subsection{Adaptive Uzawa algorithm}
\label{subsec:adaptive algorithm}
We investigate the following adaptive Uzawa algorithm, which goes back to \cite[Section~7]{ks}.

\begin{algorithm}\label{algorithm:uzawa}
\textsc{Input:} Conforming initial triangulation $\PP_0 := \TT_{000} := \TT_{\rm init}$ of $\Omega$, initial approximation $P_{00} = 0$, counters $i = j = k = 0$, adaptivity parameters $0\le\kpuno< 1$, $0 < \kpdue < 1$, $0 < {\kptre} < 1$, $0 < \vartheta \le 1$, $0<\theta\le1$, and $\Cmark\ge1$.

\noindent
\textsc{Loop:} Iterate the following steps~{\rm(i)--(iv)}:
\begin{itemize}

\item[\rm(i)]   Compute $\U_{ijk}\in \V_{ijk}$ as well as (all local contributions of) the corresponding error estimator $\eta_{ijk} = \eta(\TT_{ijk};\U_{ijk},P_{ij})$ such that the exact Galerkin approximation $\gal{ijk}[ij]\in\V_{ijk}$ of $\u_{ij}$ satisfies that $\norm{\gal{ijk}[ij]-\U_{ijk}}{\V}\le  \kpuno\,\eta_{ijk}$.

\item[\rm(ii)] {\tt while} \quad $\eta_{ijk} + \norm{\Pi_i \div \U_{ijk}}{\Omega}
\le \kpdue \, \big( \eta_{ijk} + \norm{\div \U_{ijk}}{\Omega} \big)$ \quad {\tt do}

\begin{itemize}
\item[$\bullet$] Determine $\PP_{i+1} := \binev(\PP_i,  \TT_{ijk},   \U_{ijk};\vartheta)$ by Algorithm~\ref{function:refinePressure}.
\item[$\bullet$] Define $P_{(i+1)0} := P_{ij}$, and $\TT_{(i+1)00} := \TT_{ijk}$.
\item[$\bullet$] 
Update counters $(i,j,k) \mapsto (i+1,0,0)$.
\end{itemize}
{\tt end while}

\item[\rm(iii)] {\tt if} \quad $\eta_{ijk} \le \kptre \, \norm{\Pi_i \div \U_{ijk}}{\Omega}$ \quad {\tt then}

\begin{itemize}
\item[$\bullet$] Define $P_{i(j+1)} := P_{ij} - \Pi_i \div \U_{ijk}\in\P_i$, and $\TT_{i(j+1)0} := \TT_{ijk}$.
\item[$\bullet$] 
Update counters $(i,j,k) \mapsto (i,j+1,0)$.
\end{itemize}

\item[\rm(iv)] {\tt else}

\begin{itemize}
\item[$\bullet$] Determine a set $\MM_{ijk} \subseteq \TT_{ijk}$ of (up to the fixed factor $\Cmark$) minimal cardinality, which satisfies the D\"orfler marking criterion
\begin{align}\label{eq:Doerfler}
 \theta \, \eta_{ijk}^2 \le \eta(\MM_{ijk};P_{ij},\U_{ijk})^2.
\end{align}
\item[$\bullet$] Generate $\TT_{ij(k+1)} := {\tt refine}(\TT_{ijk},\MM_{ijk})$.
\item[$\bullet$] Update counters $(i,j,k) \mapsto (i,j,k+1)$.
\end{itemize}
{\tt end if}\qed
\end{itemize}
\end{algorithm}

\begin{remark}\label{remark:single_index}
 The actual implementation of Algorithm~\ref{algorithm:uzawa} will replace the triple indices $(i,j,k)$ by one single index $n\in\N_0$, which is increased in each step~{\rm(ii)--(iv)}. However, the present statement of the algorithm makes the numerical analysis more accessible.\qed
\end{remark}

\begin{lemma}\label{lemma:set:Q}
Define the index set $\QQ := \set{(i,j,k) \in \N_0^3}{\U_{ijk} \text{ is defined by Algorithm~\ref{algorithm:uzawa}}}$. Then, for $(i,j,k) \in \N_0^3$, there hold the following assertions~{\rm(a)--(c)}:
\begin{itemize}
\item[\rm(a)] If $(i,j,k+1) \in \QQ$, then $(i,j,k) \in \QQ$.
\item[\rm(b)] If $(i,j\!+\!1,0) \in \QQ$, then $(i,j,0) \in \QQ$ and $\kk(i,j) := \max\set{k\in\N_0\!\!}{\!(i,j,k) \in \QQ} < \infty$.
\item[\rm(c)] If $(i+1,0,0) \in \QQ$, then $(i,0,0) \in \QQ$ and $\jj(i) := \max\set{j\in\N_0}{(i,j,0) \in \QQ} < \infty$.
\end{itemize}
Throughout, we shall make the following conventions for the triple index: If we write $\eta_{ij\kk}$ etc.\ (see, e.g., Lemma~\ref{lemma:convergence:j}), then (implicitly) $\kk = \kk(i,j)$. If we write $\eta_{i\jj\kk}$ etc.\ (see, e.g., Lemma~\ref{lemma:convergence:i}),  then (implicitly) $\jj = \jj(i)$ and $\kk = \kk(i,\jj)$.
\end{lemma}
\begin{proof}
Each step~{\rm(ii)--(iv)} of the algorithm increases either $i$ or $j$ or $k$ by one.
\end{proof}

\begin{remark}%
Unlike the algorithm from~\cite{ks}, our formulation of the adaptive Uzawa algorithm  avoids any special treatment of the data oscillations (i.e., to resolve $\f$ by a piecewise polynomial in an additional  loop).
\qed
\end{remark}

\begin{remark}
We note that the choice $\U_{ijk} := \gal{ijk}[ij]$ (i.e., $\kpuno = 0$) is admissible in step~{\rm (i)}  of Algorithm~\ref{algorithm:uzawa}.
In the spirit of~\cite{abem+solve}, one can also employ the PCG algorithm~\cite[Algorithm~11.5.1]{matcomp} with optimal preconditioner.
With $\kappa_1'$ and an additional index $\ell \in \N_0$ for the PCG iteration and initially $\ell := 0$, repeat the following three steps, until $U_{ijk} := U_{ijk(\ell+1)}$ satisfies that $\norm{\U_{ijk(\ell+1)}-\U_{ijk\ell}}{\V}\le \kpuno'\,\eta_{ijk(\ell+1)}$:
\begin{itemize}
\item[$\bullet$] Do one PCG step to obtain $\U_{ijk(\ell+1)}\in\V_{ijk}$ from $\U_{ijk\ell}\in \V_{ijk}$.
\item[$\bullet$] Compute (all local contributions of) the estimator $\eta_{ijk(\ell+1)} := \eta(\TT_{ijk};\U_{ijk(\ell+1)},P_{ij})$.
\item[$\bullet$] Update counters $(i,j,k,\ell) \mapsto (i,j,k,\ell+1)$.
\end{itemize}
If the preconditioner is optimal, i.e., the preconditioned linear system has uniformly bounded condition number, then it follows that PCG is a uniform contraction~\cite[Section~2.6]{abem+solve}: There exists $0 < q_{\rm pcg} < 1$ such that
\begin{align*}
 \norm{U_{\TT_{ijk}}[P_{ij}] - \U_{ijk(\ell+1)}}{\V}
 \le q_{\rm pcg} \, \norm{U_{\TT_{ijk}}[P_{ij}] - \U_{ijk\ell}}{\V}
 \quad \text{for all } \ell \in \N_0.
\end{align*}
Hence, the PCG loop terminates, and the triangle inequality proves that
\begin{align*}
 \norm{U_{\TT_{ijk}}[P_{ij}] - \U_{ijk(\ell+1)}}{\V}
 \le \frac{q_{\rm pcg}}{1 - q_{\rm pcg}} \, \norm{\U_{ijk(\ell+1)}-\U_{ijk\ell}}{\V}
 \le \frac{q_{\rm pcg}}{1 - q_{\rm pcg}} \, \kpuno' \, \eta_{ijk(\ell+1)},
\end{align*}
i.e., the criterion of step~{\rm(i)} of Algorithm~\ref{algorithm:uzawa} is satisfied for $\kpuno := \kpuno' q_{\rm pcg} / ( 1 - q_{\rm pcg} )$.\qed
\end{remark}%


\section{Convergence}\label{sec:convergence}

\subsection{Main theorem on linear convergence}\label{sec:convergence thm}

To state linear convergence, we need an ordering of the set $\QQ$ from Lemma~\ref{lemma:set:Q}: For $(i,j,k), (i',j',k') \in \QQ$, write $(i',j',k') < (i,j,k)$ if the index $(i',j',k')$ appears earlier in Algorithm~\ref{algorithm:uzawa} than $(i,j,k)$.
Define
\begin{align}
 |(i,j,k)| := \#\set{(i',j',k') \in \QQ}{(i',j',k') < (i,j,k)} \in \N_0.
\end{align}
Note that $|(i,j,k)|$ coincides with the single index $n$ from Remark~\ref{remark:single_index}.
Then, we have the following theorem. The proof is given in Section~\ref{sec:proof of linconv}. 

\begin{theorem}\label{theorem:linearconvergence}
Let $0<\kpuno<\theta^{1/2}/\Cstab$. 
Suppose that $0 < \kpdue, \kptre < 1$ are sufficiently small as in  Lemma~\ref{lemma:convergence:j} and  Lemma~\ref{lemma:convergence:i} below. Let $0<\vartheta \le 1$ and $0<\theta\le 1$. 
Then, there exist constants $\Clin > 0$ and $0 < \qlin < 1$ such that
\begin{align}\label{eq:linconv}
 \eta_{ijk} + \norm{\div \U_{ijk}}{\Omega}
 \le \Clin \qlin^{|(i,j,k)|-|(i',j',k')|} \, \big( \eta_{i'j'k'} + \norm{\div \U_{i'j'k'}}{\Omega} \big)
\end{align}
for all $(i',j',k'), (i,j,k) \in \QQ$ with $(i',j',k') < (i,j,k)$. The constants $\Clin$ and $\qlin$ depend only on  
the domain $\Omega$,  $\gamma$-shape regularity, the polynomial degree $m$, and the parameters $\kpuno, \kpdue$, $\kptre$, $\vartheta$, and $\theta$.
\end{theorem}%

\begin{remark}
The adaptive Uzawa algorithm from~\cite{bmn02} employs only one triangulation for both, the pressure and the velocity.
Similarly, we can additionally update $\PP_i:=\TT_{ij(k+1)}$ in step~{\rm(iv)} of Algorithm~\ref{algorithm:uzawa}.
Since $0<\kpdue<1$ and $\Pi_i\div U_{ijk}=\nabla U_{ijk}$, then the condition in  {\rm(ii)} will always fail.
We note that the convergence analysis of Section~\ref{sec:linconv} and in particular, linear convergence (Theorem~\ref{theorem:linearconvergence}) clearly remain valid for this modified algorithm, while our proof of optimal convergence rates (Theorem~\ref{thm:optimal}) fails.\qed
\end{remark}

\subsection{Auxiliary results}\label{sec:linconv}

The first lemma provides links between the exact Galerkin solutions $\gal{ijk}[ij]$ and its approximations $\U_{ijk}$.

\begin{lemma}\label{lemma:equivalent est}
Let $(i,j,k)\in\QQ$. 
For all $\mathcal{S}\subseteq\TT_{ijk}$, it holds that 
\begin{align}\label{eq:stability2}
|\eta(\mathcal{S};\gal{ijk}[ij],P_{ij})-\eta(\mathcal{S};\U_{ijk},P_{ij})|\le \kpuno\,\Cstab \,\eta_{ijk}, 
\end{align}
where $\Cstab>0$ is the constant from Lemma~\ref{lemma:stability}.
This particularly yields the equivalence
\begin{align}\label{eq:equivalent est}
(1-\kpuno\,\Cstab)\,\eta_{ijk}\le \etagal{ijk}[ij]\le (1+\kpuno \,\Cstab) \,\eta_{ijk}.
\end{align}
as well as the reliability estimates
\begin{align}\label{eq2:reliability:velocity}
\norm{\u_{ij}-\U_{ijk}}{\V}&\le \Crel'(\kpuno) \, \eta_{ijk},
\\
\label{eq2:reliability:reduced}
\norm{\u_i-\U_{ijk}}{\V}+\norm{p_i-P_{ij}}{\P}&\le \Crel'(\kpuno)\,\big(\eta_{ijk}+\norm{\Pi_i\div\U_{ijk}}{\Omega}\big),
\\
\label{eq2:reliability:stokes}
\norm{\u-\U_{ijk}}{\V}+\norm{p-P_{ij}}{\P}&\le \Crel'(\kpuno) \,\big(\eta_{ijk}+\norm{\div\U_{ijk}}{\Omega}\big),
\end{align}
where $\Crel'(\kpuno):=((1+\kpuno\Cstab)\Crel+\kpuno(\Crel+1))\ge \Crel$ with  $\Crel>0$ from Lemma~\ref{lem:reliability}.
\end{lemma}

\begin{proof}
To shorten notation, we set $\etagals{ijk}[ij]:=\etagal{ijk}[ij]$.
The stability \eqref{eq:stability2} follows from  Lemma~\ref{lemma:stability} and $\norm{\gal{ijk}[ij]-\U_{ijk}}{\V}\le \kpuno\, \eta_{ijk}$, which is guaranteed by step~(i) of Algorithm~\ref{algorithm:uzawa}.
Taking $\mathcal{S}=\TT_{ijk}$, \eqref{eq:equivalent est} is an immediate consequence.
To see \eqref{eq2:reliability:velocity}, we use reliability \eqref{eq:reliability:velocity}, step~(i) of Algorithm~\ref{algorithm:uzawa}, and \eqref{eq:equivalent est} to see that
\begin{align*}
\norm{\u_{ij}-\U_{ijk}}{\V}\reff{eq:reliability:velocity}\le \Crel \,\etagals{ijk}[ij]+ \norm{\gal{ijk}[ij]-\U_{ijk}}{\V} \reff{eq:equivalent est}\le ((1+\kpuno\Cstab)\Crel+\kpuno) \,\eta_{ijk}.
\end{align*}
To prove \eqref{eq2:reliability:reduced}, we apply \eqref{eq:reliability:reduced}
\begin{align*}
\norm{\u_i-\U_{ijk}}{\V}+\norm{p_i-P_{ij}}{\P}&\reff{eq:reliability:reduced}\le \Crel \big(\etagals{ijk}[ij]+\norm{\Pi_i\div\gal{ijk}[ij]}{\Omega}\big)+\norm{\gal{ijk}[ij]-\U_{ijk}}{\V} \\
&\reff{eq:equivalent est}\le ((1+\kpuno\Cstab)\Crel+\kpuno) \,\eta_{ijk}+\Crel \,\norm{\Pi_i\div\U_{ijk}}{\Omega}.
\end{align*}
 Similarly, \eqref{eq2:reliability:stokes} follows  from \eqref{eq:reliability:stokes}.
\end{proof}

The following three lemmas prove that Algorithm~\ref{algorithm:uzawa} leads to contraction if either $i$, $j$, or $k$ is increased. 
Throughout, let $0<\vartheta\le1$, $0<\theta\le1$, and, if not stated otherwise, $0\le\kpuno<1$, $0<\kpdue, \kptre<1$.

\begin{lemma}\label{lemma:convergence:k}
Let $(i,j,0) \in \QQ$ and define $\kk := \max\set{k\in\N_0\!\!}{\!(i,j,k) \in \QQ} \in \N_0 \cup \{\infty\}$.
If $0\le\kpuno<\theta^{1/2}/\Cstab$, then, there exist constants $0 < \qk < 1$ and $\Ck > 0$, which depend only on $\gamma$-shape regularity,
the polynomial degree $m$, $\kpuno$, and  $\theta$, such that
\begin{align}\label{eq0:lemma:convergence:k}
 \eta_{ij(k+n)}
 \le \Ck \, \qk^n \, \eta_{ijk} 
 \quad \text{for all } k,n \in \N_0 \text{ with } k \le k+n \le \kk.
\end{align}                                                                                                                                                                                                                                                                                                                                                                                        
Moreover, it holds that
\begin{align}\label{eq1:lemma:convergence:k}
 \eta_{ijk} \le \eta_{ijk} + \norm{\div \U_{ijk}}{\Omega}
 \le \frac{1}{\kpdue} \, \Big(1 + \frac{1}{\kptre}\Big) \, \eta_{ijk} 
 \quad \text{for all } 0 \le k < \kk.
\end{align}
If $\kk = \infty$, this yields that $\norm{\u-\U_{ijk}}{\V} + \norm{p-P_{ij}}{\P}\to0$ as $k \to \infty$ with~$p = p_i = P_{ij}$.
\end{lemma}

\begin{proof}
We split the proof into three steps.

{\bf Step~1.} 
If $\U_{ijk}=\gal{ijk}[ij]$ for all $(i,j,k)\in\QQ$, step~(iv) of Algorithm~\ref{algorithm:uzawa} is the usual adaptive step in an adaptive algorithm for, e.g., the (vector-valued) Poisson model problem. 
Hence,~\eqref{eq0:lemma:convergence:k} follows from reliability~\eqref{eq:reliability:velocity}, stability~\eqref{eq:stability} and reduction~\eqref{eq:reduction}; see, e.g., ~\cite[Theorem~4.1 (i)]{axioms}. 
For general $\U_{ijk}$,~\eqref{eq0:lemma:convergence:k} follows from \cite[Theorem~7.2]{axioms} under the constraint $0\le \kpuno<\theta^{1/2}/\Cstab$.

{\bf Step~2.}
If $k < \kk$, the 
 structure of Algorithm~\ref{algorithm:uzawa} implies that the conditions in step~(ii) and~(iii) are false, i.e.,
\begin{align*}
 \eta_{ijk} + \norm{\Pi_i \div \U_{ijk}}{\Omega}
 > \kpdue \, \big( \eta_{ijk} + \norm{\div \U_{ijk}}{\Omega} \big)
 \quad\text{and}\quad
 \eta_{ijk} > \kptre \, \norm{\Pi_i \div \U_{ijk}}{\Omega}.
\end{align*}
Hence,
\begin{align*}
 \eta_{ijk} \le \eta_{ijk} + \norm{\div \U_{ijk}}\Omega
 < \frac{1}{\kpdue} \, \big( \eta_{ijk} + \norm{\Pi_i \div \U_{ijk}}\Omega \big)
 < \frac{1}{\kpdue} \, \Big( 1 + \frac{1}{\kptre} \Big) \, \eta_{ijk}
\end{align*}
which proves~\eqref{eq1:lemma:convergence:k}.

{\bf Step~3.}
For $\kk = \infty$, the estimates~\eqref{eq0:lemma:convergence:k}--\eqref{eq1:lemma:convergence:k} imply that 
\begin{align*}
 \norm{\u-\U_{ijk}}{\V}ÃÂÃÂÃÂÃ + \norm{p - P_{ij}}{\P} 
 \stackrel{\eqref{eq2:reliability:stokes}}{\lesssim}\eta_{ijk} + \norm{\div \U_{ijk}}{\Omega}
\reff{eq1:lemma:convergence:k}\simeq \eta_{ijk}
 \xrightarrow{k \to \infty} 0.
\end{align*}
Note that $\kk = \infty$ also implies that neither $i$ nor $j$ are increased, i.e., $P_{ij}$ remains constant as $k\to\infty$. Hence, $p = P_{ij} \in \P_i$ and therefore also $p = p_i$.
\end{proof}

\begin{lemma}\label{lemma:convergence:j}
Let $(i,0,0) \in \QQ$ and define $\jj := \max\set{j\in\N_0}{(i,j,0)\in\QQ} \in \N_0 \cup \{\infty\}$.
If $0<\kptre\ll1$ is sufficiently small (see~\eqref{eq:kappa prime constraints} in the proof below), then there exist constants $0 < \qj < 1$ and $\Cj>0$ such that 
\begin{align}\label{eq0:lemma:convergence:j}
 \norm{p_i-P_{i(j+n)}}{\P} 
 \le \qj^n \, \norm{p_i - P_{ij}}{\P}
 \quad \text{for all } j,n \in \N_0 \text{ with } j \le j+n \le \jj.
\end{align}
Moreover, it holds that
\begin{align}\label{eq1:lemma:convergence:j}
 \Cj^{-1} \, \norm{p_i - P_{ij}}{\P}
 \le \eta_{ij\kk} + \norm{\div \U_{ij\kk}}{\Omega}
 \le \Cj \, \norm{p_i - P_{ij}}{\P}
 \quad \text{for all } 0\le j < \jj.
\end{align}
If $\jj = \infty$, this yields convergence $\norm{\u-\U_{ij\kk}}{\V} + \norm{p-P_{ij}}{\P}\to0$ as $j \to \infty$.
While $q_2$ depends only on 
the domain $\Omega$, $\gamma$-shape regularity,  $\kpuno$, and $\kptre$, the constant $C_2$ depends additionally on $\kpdue$.
\end{lemma}

\begin{proof}
We split the proof into three steps.

{\bf Step~1.} If $j <
\underline{j} (i)$ and $k = \underline{k} (i,j)$, then necessarily
$\underline{k} (i,j) < \infty$. The 
 structure of 
Algorithm~\ref{algorithm:uzawa} implies that the condition in
step~(ii) is false, while the condition in step~(iii) is true, i.e.,
\begin{align}\label{dpr:proof:lemma:j}
  \eta_{ij \underline{k}} +\| {\Pi_i} \nabla \cdot
  {\U}_{ij \underline{k}} \|_{\Omega} > \, \kpdue \left( \eta_{ij
  \underline{k}} +\| \nabla \cdot {\U}_{ij \underline{k}} \|_{\Omega} 
  \right) \quad \text{{{and}}} \quad \eta_{ij \underline{k}}
  \le \kptre \, \| \Pi_{i} \nabla \cdot
  {\U}_{ij \underline{k}} \|_{\Omega}. 
\end{align}
First, this proves that 
\begin{align}\label{eq:equivalence for j}
\begin{split}
&\kpdue
\, \left( \eta_{ij \underline{k}} +\| \nabla \cdot
{\U}_{ij \underline{k}} \|_{\Omega} \right) <
\eta_{ij \underline{k}} +\| {\Pi_i} \nabla \cdot
{\U}_{ij \underline{k}} \|_{\Omega} 
\le (1 + \kptre)
\, \| {\Pi_i} \nabla \cdot {\U}_{ij
\underline{k}} \|_{\Omega} \\
&\quad\le (1 + \kptre) \, \| \nabla
\cdot {\U}_{ij \underline{k}} \|_{\Omega} \le (1 + \kptre)
\, \left( \eta_{ij \underline{k}} +\| \nabla \cdot
{\U}_{ij \underline{k}} \|_{\Omega} \right).
\end{split}
\end{align}
Second, reliability~\eqref{eq2:reliability:velocity} gives that 
\begin{align} \label{eq:estimatestemp2}
  \| {\Pi_i} \nabla \cdot \left( {\u}_{i_{} j} -{\U}_{ij
  \underline{k}}  \right) \|_{\Omega} \,\le\ \|
   {\u}_{i_{} j}-{\U}_{ij \underline{k}}  \|_{\mathbb{V}}
  \overset{(\ref{eq2:reliability:velocity})}{\le}\Crel'(\kpuno)\,\eta_{ijk}
  \overset{(\ref{dpr:proof:lemma:j})}{\le}
  \kptre\Crel'(\kpuno) \, \| {\Pi_i} \nabla
  \cdot {\U}_{ij \underline{k}} \|_{\Omega}.
\end{align}
The triangle inequality yields that 
\begin{align}   \label{eq:tempjjnew}
  (1 - \kptre\Crel'(\kpuno))\, \| {\Pi_i} \nabla \cdot
  {\U}_{ij \underline{k}} \|_{\Omega} \stackrel{\eqref{eq:estimatestemp2}}\le \|
  {\Pi_i} \nabla \cdot {\u}_{i_{} j} \|_{\Omega}
  \stackrel{\eqref{eq:estimatestemp2}}\le (1 + \kptre\Crel'(\kpuno)) \| {\Pi_i} \nabla
  \cdot {\U}_{ij \underline{k}} \|_{\Omega}.
\end{align}
This leads us to 
\begin{align}  \label{eq:estimatetempgiov}
\begin{split}
&\Cdiv^{-1} \, \frac{1 - \kptre\Crel'(\kpuno)}{1 + \kptre\Crel'(\kpuno)} \|p_i - P_{ij}
\,  \|_{\mathbb{P}}  \overset{(\ref{eq:estimator:div})}{\le} 
                                  \frac{1 - \kptre\Crel'(\kpuno)}{1 + \kptre\Crel'(\kpuno)}   \,\|
  {\Pi_i} \nabla \cdot {\u}_{ij} \|_{\Omega}
\\  &\quad \overset{(\ref{eq:tempjjnew})}{\le}     
  (1 - \kptre\Crel'(\kpuno))  \, \|
  {\Pi_i} \nabla \cdot {\U}_{ij \underline{k}}
  \|_{\Omega} 
   \overset{(\ref{eq:tempjjnew})}{\le}    \| {\Pi_i} \nabla \cdot
  {\u}_{i_{} j} \|_{\Omega}  \overset{(\ref{eq:estimator:div})}{\le}   \|p_i - P_{ij} \|_{\mathbb{P}}. 
  \end{split}
\end{align}
If $\kptre\Crel'(\kpuno)<1$, the combination of ~\eqref{eq:estimatetempgiov} and ~\eqref{eq:equivalence for j} proves~\eqref{eq1:lemma:convergence:j}.

{\bf{Step~2.}} Starting from $P_{ij}$, one step of the {\emph{exact}}
Uzawa iteration for the {\emph{reduced}} Stokes problem (leading to the
auxiliary quantity $p_{i (j + 1)}$) guarantees the 	existence of some $0 < q_{\rm{Uzawa}} < 1$ such that the following contraction holds (see~[KS08,~Eq.~(4.3)]):
\begin{align}\label{eq:discrete contraction}
 \|p_i - p_{i (j + 1)} \|_{\mathbb{P}} \, \le \,
   q_{\rm{Uzawa}} \, \|p_i - P_{ij} \|_{\mathbb{P}} 
   \quad \text{{ with}} \quad 
    p_{i (j + 1)} \:= P_{ij} -
   {\Pi_i} \nabla \cdot {\u}_{ij} .
\end{align}
The contraction constant
$q_{\rm{Uzawa}}$ is the norm of the
operator from~\eqref{eq:N_alpha} with $\alpha=1$.
Indeed, the proof of \eqref{eq:discrete contraction} works exactly as in Appendix~\ref{appendix:operator} if $S:\P\to\P$ is replaced by the operator $\Pi_iS:\P_i\to\P_i$.
In particular, $q_{\mathrm{\rm{Uzawa}}}$
does neither depend on $i$ nor on $j$. Since $P_{i (j + 1)} = P_{ij} -
{\Pi_i} \nabla \cdot {\U}_{ij \underline{k}}$, we
are thus led to
\begin{align*}
  \|p_i - P_{i (j + 1)} \|_{\mathbb{P}} &\, \le  \, \|p_i - p_{i (j + 1)}
  \|_{\mathbb{P}} +\|p_{i (j + 1)} - P_{i (j + 1)} \|_{\mathbb{P}} \\
  &\, \le \, q_{\mathrm{\rm{Uzawa}}} \, \|p_i - P_{ij}
  \|_{\mathbb{P}} \, +\| {\Pi_i} \nabla \cdot
  ({\u}_{ij} - {\U}_{ij \underline{k}}) \|_{\mathbb{P}}
  \\
  &~\reff{eq:estimatestemp2}{\le} 
  q_{\mathrm{\rm{Uzawa}}} \, \|p_i - P_{ij} \|_{\mathbb{P}} \,
  + \kptre\Crel'(\kpuno) \, \| {\Pi_i} \nabla
  \cdot {\U}_{ij \underline{k}} \|_{\Omega} \\
  &~\reff{eq:estimatetempgiov}{\le} 
  \left(q_{\mathrm{\rm{Uzawa}}} + \frac{\kptre\Crel'(\kpuno) }{ 1 - \kptre\Crel'(\kpuno) }\right) \, \|p_i -
  P_{ij} \|_{\mathbb{P}}=:q_2\,\norm{p_i-P_{ij}}{\P}.
\end{align*}
Let $0 < \kptre \ll 1$ be sufficiently small, i.e.,
\begin{align}\label{eq:kappa prime constraints}
0< \kptre\Crel'(\kpuno) <1\quad\text{and}\quad  0<q_2:=q_{\rm Uzawa} +\frac{\kptre\Crel'(\kpuno)}{1-\kptre\Crel'(\kpuno)}<1. 
   \end{align}
Then, induction proves that $\|p_i - P_{i (j + n)} \|_{\mathbb{P}} \, \le
q_2^n \, \|p_i - P_{ij} \|_{\mathbb{P}}$ for every $j, n \in
\mathbb{N}_0$ with $j \le j + n \le \underline{j}$. 
This
proves~(\ref{eq0:lemma:convergence:j}).

{\bf{Step~3.}} For $\underline{j} = \infty$, the
estimates~(\ref{eq0:lemma:convergence:j})--(\ref{eq1:lemma:convergence:j})
imply that
\begin{align*} 
\| \u - {\U}_{ij \underline{k}} \|_{\mathbb{V}} +\|p -P_{ij} \|_{\mathbb{P}} 
\reff{eq2:reliability:stokes}\lesssim
   \eta_{ij \underline{k}} +\| \nabla \cdot {\U}_{ij \underline{k}}
   \|_{\Omega}
   {\overset{(\ref{eq1:lemma:convergence:j})}{\simeq}} \|p_i -
   P_{ij} \|_{\mathbb{P} } \xrightarrow{j \to \infty} 0. \end{align*}
This concludes the proof.
\end{proof}

Note that $\ii := \max \set{i \in \N_0}{(i,0,0) \in \QQ} < \infty$ in Algorithm~\ref{algorithm:uzawa} implies that either $\jj := \jj(\ii) = \infty$ or $\kk(\ii,\jj) = \infty$. According to Lemma~\ref{lemma:convergence:k} (for $\kk=\infty$) and Lemma~\ref{lemma:convergence:j} (for $\jj=\infty$), it only remains to analyze the case $\ii = \infty$.

\begin{lemma}\label{lemma:convergence:i}
Let $\ii := \max \set{i \in \N_0}{(i,0,0) \in \QQ} \in \N_0 \cup \{\infty\}$. If $0 < \kpdue \ll 1$ is sufficiently small (see~\eqref{eq:CCC} in the proof below), then there exist constants $0 < \qi < 1$ and $\Ci > 0$ such that
\begin{align}\label{eq0:lemma:convergence:i}
 \norm{p - P_{(i+n)\jj}}{\P} \le \qi^n \, \norm{p - P_{i\jj}}{\P}
 \quad \text{for all } i, n \in \N_0 \text{ with } i\le i+n \le \ii.
\end{align}
Moreover, it holds that 
\begin{align}\label{eq1:lemma:convergence:i}
 \Ci^{-1} \, \norm{p - P_{i\jj}}{\P}
 \le \eta_{i\jj\kk} + \norm{\div \U_{i\jj\kk}}{\Omega} 
 \le \Ci \, \norm{p - P_{i\jj}}{\P}
 \quad \text{for all } 0\le i < \ii.
\end{align}
While $\Ci$ depends only on 
the domain $\Omega$, $\gamma$-shape regularity, $\kpuno$ and  $\kpdue$, the contraction constant $\qi$ depends additionally on $0<\vartheta\le1$.
If $\ii = \infty$, this yields convergence $\norm{\u - \U_{i\jj\kk}}{\V} + \norm{p - P_{i\jj}}{\P} \to 0$ as $i \to \infty$.
\end{lemma}

\begin{proof}
We split the proof into five steps.

{\bf Step~1.}
According to Algorithm~\ref{algorithm:uzawa}, it holds that
\begin{align}\label{eq1:proof:lemma:i}
 \eta_{i\jj\kk} + \norm{\Pi_i \div \U_{i\jj\kk}}{\Omega}
 \le \kpdue \, \big( \eta_{i\jj\kk} + \norm{\div \U_{i\jj\kk}}{\Omega} \big).
\end{align}
For $0 < \kpdue < 1$, this implies that 
\begin{align*}
 \eta_{i\jj\kk} + \norm{\Pi_i \div \U_{i\jj\kk}}{\Omega} \le \frac{\kpdue}{1-\kpdue} \, \norm{\div \U_{i\jj\kk}}{\Omega}.
\end{align*}
Recall that 
\begin{align*}
 \norm{\div \U_{i\jj\kk}}{\Omega}
 \,\le  \norm{\div \u_{i\jj}}{\Omega} +\norm{\div (\u_{i\jj} - \U_{i\jj\kk})}{\Omega}
\reff{eq2:reliability:velocity}\le   \norm{\div \u_{i\jj}}{\Omega}+\Crel'(\kpuno)\,\eta_{i\jj\kk}
\end{align*}
We abbreviate $C(\kpuno, \kpdue):=\Crel'(\kpuno)\,\kpdue/(1-\kpdue)$.
For sufficiently small $0 < \kpdue \ll 1$ with $0<C(\kpuno, \kpdue)<1$, the combination of the last two estimates implies that
$\norm{\div \U_{i\jj\kk}}{\Omega} \le (1-C(\kpuno, \kpdue))^{-1}\,\norm{\div \u_{i\jj}}{\Omega}$.
With 
\begin{align*}
C'(\kpuno, \kpdue):=\frac{{C(\kpuno, \kpdue)}}{{1-C(\kpuno, \kpdue)}},
\end{align*} we are hence led to
\begin{align}\label{eq2:proof:lemma:i}
\begin{split}
 &\norm{\u_{i\jj} - \U_{i\jj\kk}}{\V}
 \reff{eq2:reliability:velocity}\le \Crel'(\kpuno)
\,\big( \eta_{i\jj\kk} +\norm{\Pi_i\div\U_{i\jj\kk}}{\Omega}\big)
 \le
 C(\kpuno, \kpdue) \, \norm{\div \U_{i\jj\kk}}{\Omega} 
 \\& \qquad
 \,\le\, C'(\kpuno, \kpdue) \, \norm{\div \u_{i\jj}}{\Omega}
~\reff{eq2:estimator:div}\le C'(\kpuno, \kpdue) \, \norm{p - P_{i\jj}}{\P}.
\end{split}
\end{align}
Conversely,
\begin{align*}
 \norm{p - P_{i\jj}}{\P}
 &\reff{eq2:estimator:div}\le\Cdiv \norm{\div \u_{i\jj}}{\Omega}
 \le\Cdiv \big(\norm{\div \U_{i\jj\kk}}{\Omega} + \norm{\div (\u_{i\jj} - \U_{i\jj\kk})}{\Omega}\big)\\
 &\reff{eq2:reliability:velocity}\le \max\{1,\Crel'(\kpuno)\} \, \Cdiv \, \big( \norm{\div \U_{i\jj\kk}}{\Omega} + \eta_{i\jj\kk}\big).
\end{align*}
In particular, this proves~\eqref{eq1:lemma:convergence:i}.

{\bf Step~2.}
Recall from Step~1 that
\begin{align}\label{eq3:proof:lemma:i}
\begin{split}
 \norm{\div (\u_{i\jj} - \U_{i\jj\kk})}{\Omega} 
 + \norm{\Pi_i \div \U_{i\jj\kk}}{\Omega}
 &\reff{eq2:reliability:velocity}\le \max\{1,\Crel'(\kpuno)\}\, \big(\eta_{i\jj\kk} 
 + \norm{\Pi_i \div \U_{i\jj\kk}}{\Omega}\big)
 \\&
\reff{eq2:proof:lemma:i}\le 
\max\{1,\Crel'(\kpuno)\} \, C'(\kpuno, \kpdue)  \, \norm{p - P_{i\jj}}{\P}.
\end{split}
\end{align}
We hence observe that
\begin{align*}
 \norm{p_i - P_{i\jj}}{\P} 
&~\reff{eq:estimator:div} \le \Cdiv \norm{\Pi_i \div \u_{i\jj}}{\Omega}
 \le \Cdiv\big( \norm{\Pi_i \div (\u_{i\jj} - \U_{i\jj\kk})}{\Omega}
 + \norm{\Pi_i \div \U_{i\jj\kk}}{\Omega}\big)\\
&~\reff{eq3:proof:lemma:i}\le \Cdiv\max\{1,\Crel'(\kpuno)\} \, C'(\kpuno, \kpdue) \, \norm{p-P_{i\jj}}{\P}.
\end{align*}

{\bf Step~3.}
From Algorithm~\ref{function:refinePressure}, we obtain that
\begin{align*}
 \vartheta \, \norm{\div \U_{i\jj\kk}}{\Omega}
 \le \norm{\Pi_{i+1} \div \U_{i\jj\kk}}{\Omega}.
\end{align*}
According to~\eqref{eq2:proof:lemma:i}, it holds that
\begin{align*}
 \norm{\div \u_{i\jj}}{\Omega}
 \le \norm{\div \U_{i\jj\kk}}{\Omega} + \norm{\div (\u_{i\jj}-\U_{i\jj\kk})}{\Omega}
 \reff{eq2:proof:lemma:i}\le (1 + C(\kpuno, \kpdue)) \, \norm{\div \U_{i\jj\kk}}{\Omega}, 
\end{align*}
as well as
\begin{align*}
 \norm{\Pi_{i+1} \div (\u_{i\jj} - \U_{i\jj\kk})}{\Omega} 
 &\le \norm{\u_{i\jj} - \U_{i\jj\kk}}{\V}
 \reff{eq2:proof:lemma:i}\le C'(\kpuno, \kpdue) \, \norm{\div \u_{i\jj}}{\Omega}.
\end{align*}
Combining the last three estimates, we see that
\begin{align*}
 \norm{\Pi_{i+1} \div \u_{i\jj}}{\Omega}
 &\ge \norm{\Pi_{i+1} \div \U_{i\jj\kk}}{\Omega}
 - \norm{\Pi_{i+1} \div ( \u_{i\jj} - \U_{i\jj\kk} )}{\Omega}
 \\&
 \ge \Big( \frac{\vartheta}{1+C(\kpuno, \kpdue)} - C'(\kpuno, \kpdue) \Big) \, \norm{\div \u_{i\jj}}{\Omega}.
\end{align*}%
Recall the constant $\Cdiv\ge1$ from~\eqref{eq:Uzapape2.2}. 
If $0 < \kpdue \ll 1$ is sufficiently small, it holds that
$C''(\kappa_1,\kappa_2,\vartheta) := \big(\frac{\vartheta}{1+C(\kpuno, \kpdue)} - C'(\kpuno, \kpdue)\big) / \Cdiv > 0$.
This implies that
\begin{align*}
 \norm{p_{i+1} - P_{i\jj}}{\P}
~\reff{eq:estimator:div}\ge  \norm{\Pi_{i+1} \div \u_{i\jj}}{\Omega}
 &\,\ge\, \Big( \frac{\vartheta}{1+C(\kpuno, \kpdue)} - C'(\kpuno, \kpdue) \Big) \, \norm{\div \u_{i\jj}}{\Omega}
 \\
 &\reff{eq2:estimator:div}\ge C''(\kappa_1,\kappa_2,\vartheta) \, \norm{p-P_{i\jj}}{\P}.
\end{align*}
Together with the Pythagoras theorem, we are hence led to
\begin{align*}
 \norm{p-p_{i+1}}{\P}^2 
 = \norm{p-P_{i\jj}}{\P}^2 - \norm{p_{i+1} - P_{i\jj}}{\P}^2
 \le (1-C''(\kappa_1,\kappa_2,\vartheta)^2) \, \norm{p-P_{i\jj}}{\P}^2.
\end{align*}

{\bf Step~4.}
Combining Step~2 and Step~3, we obtain that
\begin{align*}
& \norm{p - P_{(i+1)\jj}}{\P}^2
 = \norm{p - p_{i+1}}{\P}^2 + \norm{p_{i+1} - P_{(i+1)\jj}}{\P}^2
 \\&\quad
 \le \left(1-C''(\kappa_1,\kappa_2,\vartheta)^2\right) \, \norm{p-P_{i\jj}}{\P}^2 + \Cdiv^2 \, \max\{1,\Crel'(\kpuno)^2\} \, C'(\kpuno, \kpdue)^2 \, \norm{p - P_{(i+1)\jj}}{\P}^2.
\end{align*}
For sufficiently small $0 < \kpdue \ll 1$, i.e.,
\begin{align}\label{eq:CCC}
\begin{split}
 &C(\kpuno, \kpdue) = \frac{\Crel'(\kpuno)\kpdue}{1-\kpdue} < 1, 
\\
&0 < C''(\kappa_1,\kappa_2,\vartheta) = \Big( \frac{\vartheta}{1+C(\kpuno, \kpdue)} - \frac{C(\kpuno, \kpdue)}{1-C(\kpuno, \kpdue)} \Big) \, \Cdiv^{-1},
 \\
&0 < q_3^2 := \frac{1-C''(\kappa_1,\kappa_2,\vartheta)^2}{1-\Cdiv^2 \, \max\{1,\Crel'(\kpuno)^2\} \, C'(\kpuno, \kpdue)^2}<1,\hspace*{-3mm}
\end{split}
\end{align}%
we hence see that
\begin{align*}
 \norm{p - P_{(i+1)\jj}}{\P}^2 
 \le q_3^2 \, \norm{p-P_{i\jj}}{\P}^2.
\end{align*}
By induction, we conclude~\eqref{eq0:lemma:convergence:i}.

{\bf Step~5.}
For $\ii=\infty$, the estimates~\eqref{eq0:lemma:convergence:i}--\eqref{eq1:lemma:convergence:i} imply that 
\begin{align*}
 \norm{\u-\U_{i\jj\kk}}{\V} + \norm{p - P_{i\jj}}{\P} 
 \stackrel{\eqref{eq2:reliability:stokes}}{\lesssim} \eta_{i\jj\kk} + \norm{\div \U_{i\jj\kk}}{\Omega}
 \stackrel{\eqref{eq1:lemma:convergence:i}}\lesssim \norm{p-P_{i\jj}}{\P}
 \xrightarrow{i \to \infty} 0.
\end{align*}
This concludes the proof.
\end{proof}

\subsection{Proof of Theorem~\ref{theorem:linearconvergence}}\label{sec:proof of linconv}

To prove Theorem~\ref{theorem:linearconvergence}, we need the following two lemmas.
A slightly weaker version of the first lemma is already proved in~\cite[Lemma~4.9]{axioms}.
The elementary proof, however, immediately extends to the  following generalization and is therefore omitted.
The second lemma states certain quasi-monotonicities for the output of the adaptive algorithm.

\begin{lemma}\label{lem:uniform summability}
Let $(a_\ell)_{\ell\in\N_0}$ be a sequence with $a_\ell\ge 0$ for all $\ell\in\N_0$. 
With the convention $0^{-1/s}:=\infty$,  the following three statements are pairwise equivalent:
\begin{itemize}
\item[\rm (a)] There exist a constant $C>0$ such that 
$\sum_{n=\ell}^\infty a_n \le C a_\ell$  {for all }$\ell\in\N_0$.
\item[\rm(b)] For all $s>0$, there exists $C>0$ such that
$\sum_{n=0}^{\ell} a_n^{-1/s}\le C a_\ell^{-1/s}$  {for all }$\ell\in\N_0$.
\item[\rm(c)]\label{item:linconvtmp} There exist  $0<q<1$ and $C>0$ such that 
$a_{\ell+n}\le C q^n a_\ell$ { for all } $n,\ell\in\N_0$.
\end{itemize}
Here, in each statement, the constants $C > 0$ may differ. \hfill$\blacksquare$ 
\end{lemma}


\begin{lemma}\label{lemma:monotone}
Let $0<\kpuno<\theta^{1/2}/\Cstab$.
Suppose that $\kpdue,\kptre$ are sufficiently small as in Lemma~\ref{lemma:convergence:j} and  Lemma~\ref{lemma:convergence:i}. 
Let $(i,j,0) \in \QQ$. Then,  there hold the  assertions~{\rm(a)--(d)}:
\begin{itemize}
\item[\rm(a)] If $i\ge1$, then $\eta_{i00} + \norm{\div \U_{i00}}{\Omega} \le\Cmon\big( \eta_{(i-1)\jj\kk} + \norm{\div \U_{{(i-1)}\jj\kk}}{\Omega}\big)$.
\item[\rm(b)] If $j\ge1$, then $\eta_{ij0} + \norm{\div \U_{ij0}}{\Omega} \le \Cmon \, \big( \eta_{i(j-1)\kk} + \norm{\div \U_{i(j-1)\kk}}{\Omega} \big)$.
\item[\rm(c)] $\eta_{ijk} + \norm{\div \U_{ijk}}{\Omega} \le \Cmon \, \big( \eta_{ijk'} + \norm{\div \U_{ijk'}}{\Omega} \big)$ for all $0 \le k' \le k \le \kk(i,j)$.
\item[\rm(d)] $\eta_{ij\kk} + \norm{\div \U_{ij\kk}}{\Omega} \le \Cmon \, \big( \eta_{ij'\kk} + \norm{\div \U_{ij'\kk}}{\Omega} \big)$ for all $0 \le j' \le j < \jj(i)$.
\end{itemize}
The constant $\Cmon>0$ depends only on 
 $\Omega$, $\Cstab$, $\Crel$, $\Ck$, and $\Cj$.
\end{lemma}

\begin{proof}
To shorten notation, we set $\etagals{ijk}[ij]:=\etagal{ijk}[ij]$ and $\gals{ijk}[ij]:=\gal{ijk}[ij].$
To prove~(a), recall from step~(ii) of Algorithm~\ref{algorithm:uzawa} that $\TT_{i00} = \TT_{(i-1)\jj\kk}$ as well as $P_{i0} = P_{(i-1)\jj}$. Hence, $\gals{i00}[i0] = \gals{(i-1)\jj\kk}[(i-1)\jj]$ and consequently $\etagals{i00}[i0] = \etagals{(i-1)\jj\kk}[(i-1)\jj]$ as well as $\norm{\div \gals{i00}[i0]}{\Omega} = \norm{\div  \gals{(i-1)\jj\kk}[(i-1)\jj]}{\Omega}$.
Since $\kpuno<\theta^{1/2}\Cstab^{-1}\le \Cstab^{-1}$, we can apply the equivalence \eqref{eq:equivalent est} in both directions.
With step~(i) of Algorithm~\ref{algorithm:uzawa}, we see that 
\begin{align*}
&\eta_{i00}+\norm{\div\U_{i00}}{\Omega} 
\reff{eq:equivalent est}\lesssim \etagals{i00}[i0] +\norm{\div\gals{i00}[i0]}{\Omega}+\norm{\gals{i00}[i0]-\U_{i00}}{\V}
\lesssim \etagals{i00}[i0] +\norm{\div\gals{i00}[i0]}{\Omega} +\eta_{i00}
\\
&\quad\reff{eq:equivalent est}\lesssim \etagals{i00}[i0] +\norm{\div\gals{i00}[i0]}{\Omega} 
=\etagals{(i-1)\jj\kk}[(i-1)\jj]+\norm{\div\gals{(i-1)\jj\kk}[(i-1)\jj]}{\Omega}
\reff{eq:equivalent est}\lesssim  \eta_{(i-1)\jj\kk} + \norm{\div \U_{(i-1)\jj\kk}}{\Omega}
\\&\qquad+\norm{\gals{(i-1)\jj\kk}[(i-1)\jj]- \U_{(i-1)\jj\kk}}{\V}
\lesssim \eta_{(i-1)\jj\kk} + \norm{\div \U_{(i-1)\jj\kk}}{\Omega}.
\end{align*}

To prove~(b), recall from step~(iii) of Algorithm~\ref{algorithm:uzawa} that $\TT_{ij0} = \TT_{i(j-1)\kk}$ and $P_{ij} = P_{i(j-1)} - \Pi_{i} \div \U_{i(j-1)\kk}$. According to the discrete variational form~\eqref{eq:weakform:discrete}, it holds that
\begin{align*}
 a(\gals{ij0}[ij] - \gals{i(j-1)\kk}[i(j-1)] , \VV_{ij0}) = b(\VV_{ij0} , \Pi_{i} \div \U_{i(j-1)\kk})
 \quad \text{for all } \VV_{ij0} \in \V(\TT_{ij0}) = \V(\TT_{(i-1)\jj\kk}).
\end{align*}
This proves that $\norm{\gals{ij0}[ij] - \gals{i(j-1)\kk}[i(j-1)]}{\V} \lesssim \norm{\Pi_{i} \div \U_{i(j-1)\kk}}{\Omega} \le \norm{\div \U_{i(j-1)\kk}}{\Omega}$.
First, it follows that
\begin{align*}
 &\norm{\div \U_{ij0}}{\Omega} 
 \le \norm{\div \U_{i(j-1)\kk}}{\Omega} + \norm{\U_{ij0} - \U_{i(j-1)\kk}}{\V}
\le \norm{\div \U_{i(j-1)\kk}}{\Omega}+ \norm{ \gals{ij0}[ij] - \gals{i(j-1)\kk}[i(j-1)]}{\V} 
\\
&\quad+\norm{\gals{ij0}[ij]-\U_{ij0}}{\V}
+\norm{\gals{i(j-1)\kk}[i(j-1)]-\U_{i(j-1)\kk}}{\V}
 \le \norm{\div \U_{i(j-1)\kk}}{\Omega}+\kpuno\,\eta_{ij0}+\kpuno\,\eta_{i(j-1)\kk}.
\end{align*}
Second, stability of the error estimator (Lemma~\ref{lemma:stability}),  $\TT_{ij0} = \TT_{i(j-1)\kk}$ and the previous estimate prove that
\begin{align*}
 \eta_{ij0}&\reff{eq:stability}\le \eta_{i(j-1)\kk} + \Cstab \, \big(\norm{\U_{ij0} - \U_{i(j-1)\kk}}{\V}
 +\norm{\Pi_i\div\U_{i(j-1)\kk}}{\Omega}\big)
 \\
& \,\le(1+\kpuno\Cstab)\,\eta_{i(j-1)\kk} + \Cstab\,\norm{\div \U_{i(j-1)\kk}}{\Omega}+\kpuno\Cstab\,\eta_{ij0}.
\end{align*}
Recall that $\kpuno\Cstab<\theta^{1/2}\le 1$.
Thus, combining the last two estimates, we conclude the proof of~(b).

To prove~(c), note that Lemma~\ref{lemma:convergence:k} implies that
\begin{align}\label{eq:Cmon c1}
 \eta_{ijk}~\reff{eq0:lemma:convergence:k}\le \Ck \, \eta_{ijk'}
 \quad \text{for all } 0 \le k' < k \le \kk := \kk(i,j).
\end{align}
Moreover, the Pythagoras theorem, reliability~\eqref{eq:reliability:velocity}, and the equivalence~\eqref{eq:equivalent est} prove that
\begin{align*}
 \norm{\div \U_{ijk}}{\Omega} 
 &\,\le \norm{\div \U_{ijk'}}{\Omega} + \norm{ \gals{ijk}[ij] - \gals{ijk'}[ij] }{\V}
 +\norm{\gals{ijk}[ij]-\U_{ijk}}{\V}+\norm{\gals{ijk'}[ij]-\U_{ijk'}}{\V}
 \\
 &\,\le \norm{\div \U_{ijk'}}{\Omega} + \norm{\u_{ij} - \gals{ijk'}[ij]}{\V}
 +\kpuno\,\eta_{ijk}+\kpuno\,\eta_{ijk'}\\
 &\hspace{-11pt}\stackrel{\eqref{eq:reliability:velocity}+\eqref{eq:Cmon c1}}\lesssim \norm{\div \U_{ijk'}}{\Omega} + \etagals{ijk'}[ij] +\eta_{ijk'}
\\
 &\reff{eq:equivalent est}\lesssim \norm{\div \U_{ijk'}}{\Omega} + \eta_{ijk'}.
\end{align*}

To prove~(d), note that Lemma~\ref{lemma:convergence:j} implies that
\begin{align*}
 \eta_{ij\kk} + \norm{\div \U_{ij\kk}}{\Omega} 
~\reff{eq1:lemma:convergence:j}\simeq \norm{p_i-P_{ij}}{\P}
~\reff{eq0:lemma:convergence:j}\le \norm{p_i-P_{ij'}}{\P}
~\reff{eq1:lemma:convergence:j}\simeq \eta_{ij'\kk} + \norm{\div \U_{ij'\kk}}{\Omega}.
\end{align*}
This concludes the proof.
\end{proof}

\begin{proof}[Proof of Theorem~\ref{theorem:linearconvergence}]
For all $0 \le i' \le i\le \ii$, define $\jjj(i) \in \N_0$ by
\begin{align*}
 \jjj(i) := \begin{cases}
0 \quad & \text{if $i' < i$},\\
 j' \quad&\text{if $i' = i$}.
 \end{cases}
\end{align*}
For all $0\le i'\le i\le\ii$ and all $\jjj(i) \le j \le \jj(i)$, define $\kkk(i,j) \in \N_0$ by 
\begin{align*}
 \kkk(i,j) := \begin{cases}
0\quad & \text{if $i' < i$ or $j' < j$},\\
 k' \quad & \text{if $i' = i $ and $j' = j$}.
 \end{cases}
\end{align*}
As for $\jj$ and $\kk$, we write $\jjj=\jjj(i)$ and $\kkk=\kkk(i,j)$ if $i$ and $j$ are clear from the context.
Further, we abbreviate
\begin{align*}
\mu_{ijk}:=\eta_{ijk} + \norm{\div \U_{ijk}}{\Omega}.
\end{align*}
With this notation and according to Lemma~\ref{lem:uniform summability},~\eqref{eq:linconv} is equivalent to 
\begin{align}\label{eq:linconv proof}
\sum_{(i,j,k)\in\QQ\atop(i',j',k') \le (i,j,k)}  \mu_{ijk}
=\sum_{i=i'}^{\ii}\sum_{j=\jjj(i)}^{\jj(i)}\sum_{k=\kkk(i,j)}^{\kk(i,j)} \mu_{ijk}
\lesssim  \mu_{i'j'k'}
\quad \text{for all $(i',j',k')\in\QQ$.}
\end{align} 
We prove~\eqref{eq:linconv proof} in the following three steps.

{\bf Step~1.}
For $\kkk(i,j)<\kk(i,j)<\infty$, Lemma~\ref{lemma:monotone} (c)  proves that $\mu_{ij\kk}\lesssim \mu_{ij\kkk}$
Hence, Lemma~\ref{lemma:convergence:k} in combination with the geometric series 
allows to estimate the  sum over $k$
\begin{align}\begin{split}
&\sum_{i=i'}^{\ii}\sum_{j=\jjj(i)}^{\jj(i)}\sum_{k=\kkk(i,j)}^{\kk(i,j)} \mu_{ijk}\stackrel{\rm(c)}\lesssim\sum_{i=i'}^{\ii}\sum_{j=\jjj(i)}^{\jj(i)}\sum_{k=\kkk(i,j)}^{\kk(i,j)-1} \mu_{ijk}\reff{eq1:lemma:convergence:k}\simeq\sum_{i=i'}^{\ii}\sum_{j=\jjj(i)}^{\jj(i)}\sum_{k=\kkk(i,j)}^{\kk(i,j)-1} \eta_{ijk}\reff{eq0:lemma:convergence:k}\lesssim \sum_{i=i'}^{\ii}\sum_{j=\jjj(i)}^{\jj(i)} \eta_{ij\kkk}\\
&\qquad\le \sum_{i=i'}^{\ii}\sum_{j=\jjj(i)}^{\jj(i)} \mu_{ij\kkk}
= \sum_{j=\jjj(i')}^{\jj(i')} \mu_{i'j\kkk}+ \sum_{i=i'+1}^{\ii}\sum_{j=\jjj(i)}^{\jj(i)} \mu_{ij\kkk}
= \sum_{j=j'}^{\jj(i')} \mu_{i'j\kkk}+\sum_{i=i'+1}^{\ii}\sum_{j=0}^{\jj(i)} \mu_{ij0}.\label{eq:linconv terms}
\end{split}
\end{align}

 {\bf Step~2.}
In this step, we bound the first summand of~\eqref{eq:linconv terms} by $\mu_{i'j'k'}$.
It holds that 
 \begin{align*}
 \sum_{j=j'}^{\jj(i')} \mu_{i'j\kkk}= \mu_{i'j'\kkk}+\sum_{j=j'+1}^{\jj(i')} \mu_{i'j\kkk}=\mu_{i'j'k'}+\sum_{j=j'+1}^{\jj(i')} \mu_{i'j0}.
 \end{align*}
Lemma~\ref{lemma:monotone} (b) and Lemma~\ref{lemma:convergence:j} in combination with the geometric series show that
 \begin{align*}
 \sum_{j=j'+1}^{\jj(i')} \mu_{i'j0}\stackrel{\rm (b)}
 \lesssim \sum_{j=j'+1}^{\jj(i')} \mu_{i'(j-1)\kk}
 = \sum_{j=j'}^{\jj(i')-1} \mu_{i'j\kk}
~\reff{eq1:lemma:convergence:j}\simeq
 \sum_{j=j'}^{\jj(i')-1}\norm{p_{i'}-P_{i'j}}{\P}\reff{eq0:lemma:convergence:j}
 \lesssim\norm{p_{i'}-P_{i'j'}}{\P}
\reff{eq:reliability:reduced}\lesssim  \mu_{i'j'k'}.
 \end{align*}
 
{\bf Step~3.}
In this step, we bound the second summand of~\eqref{eq:linconv terms} by $\mu_{i'j'k'}$.
First, we consider only the terms where $j>0$.
As in Step~2, Lemma~\ref{lemma:monotone} (b) and Lemma~\ref{lemma:convergence:j} in combination with the geometric series show that
 \begin{align*}
 \sum_{i=i'+1}^{\ii}\sum_{j=1}^{\jj(i)} \mu_{ij0}
 \stackrel{\rm (b)}\lesssim 
 \sum_{i=i'+1}^{\ii}\sum_{j=1}^{\jj(i)} \mu_{i(j-1)\kk}
 = \sum_{i=i'+1}^{\ii}\sum_{j=0}^{\jj(i)-1} \mu_{ij\kk}
 \!\!\!\stackrel{\text{Lem.}\ref{lemma:convergence:j}}\lesssim\!\!\!
 \sum_{i=i'+1}^{\ii}\mu_{i0\kk}     \stackrel{\rm (c)}\lesssim \sum_{i=i'+1}^{\ii}\mu_{i00}.
 \end{align*}
Hence, it holds that 
\begin{align*}
 \sum_{i=i'+1}^{\ii}\sum_{j=0}^{\jj(i)} \mu_{ij0} = \sum_{i=i'+1}^{\ii}\mu_{i00}+ \sum_{i=i'+1}^{\ii}\sum_{j=1}^{\jj(i)} \mu_{ij0}\lesssim\sum_{i=i'+1}^{\ii}\mu_{i00}.
 \end{align*}
  Lemma~\ref{lemma:monotone} (a) and Lemma~\ref{lemma:convergence:i} in combination with 
the geometric series show that
 \begin{align*}
 \sum_{i=i'+1}^{\ii}\mu_{i00}\stackrel{\rm (a)}\lesssim \sum_{i=i'+1}^{\ii}\mu_{(i-1)\jj\kk} 
 = \sum_{i=i'}^{\ii-1}\mu_{i\jj\kk}
~\reff{eq1:lemma:convergence:i}\simeq
 \sum_{i=i'}^{\ii-1} \norm{p-P_{i\jj}}{\P}\reff{eq0:lemma:convergence:i}
 \lesssim \norm{p-P_{i'\jj}}{\P}
~\reff{eq:reliability:stokes}\lesssim 
 \mu_{i'\jj\kk}.
 \end{align*}
 If $j' = \jj(i')$, then Lemma~\ref{lemma:monotone} (c) yields that 
$\mu_{i'\jj\kk}= \mu_{i'j'\kk}
\lesssim \mu_{i'j'k'}$.
Otherwise, if $j' < \jj(i')$, then Lemma~\ref{lemma:monotone} (b)--(d) yield that 
 \begin{align*}
 \mu_{i'\jj\kk}
 \stackrel{\rm(c)}\lesssim \mu_{i'\jj 0}
 \stackrel{\rm(b)}\lesssim \mu_{i'(\jj-1)\kk}
 \stackrel{\rm(d)}\lesssim \mu_{i'j'\kk}
 \stackrel{\rm(c)}\lesssim \mu_{i'j'k'}.
 \end{align*}
 Altogether,  we have derived~\eqref{eq:linconv proof}, which concludes the proof.
\end{proof}%

\section{Convergence rates}\label{sec:rates}

\subsection{Main theorem on optimal convergence rates}
\label{subsec:main theorem on rates}

The first lemma relates two different characterizations of approximation classes from the literature, which are either based on the accuracy $\eps > 0$ (see, e.g.,~\cite{stevenson,ks}) or the number of elements $N$ (see, e.g.,~\cite{ckns,axioms}).

\begin{lemma}\label{lemma1:apx}
Recall that $\Tc =~ \Tc(\TT_{\rm init})$. 
Let $\varrho: \Tc \to \R_{\ge0}$ satisfy that $\inf_{\TT \in \Tc} \rho(\TT) = 0$.
Let $s > 0$ and define
\begin{align}\label{eq2:lemma:apx}
 \A_s^{\rm c}(\varrho) := \sup_{N \in \N_0} \Big( (N+1)^s  \min_{\TT \in \Tc_N} \varrho(\TT) \Big),
 \text{ where } \Tc_N := \set{\TT \in \Tc}{\#\TT - \#\TT_{\rm init} \le N}.
 \hspace*{-5mm}
\end{align}
With $\Tc_\eps(\varrho) := \set{\TT \in \Tc}{\varrho(\TT) \le \eps} \neq \emptyset$ for $\eps > 0$, there holds the equality 
\begin{align}\label{eq3:lemma:apx}
  \A_s^{\rm c}(\varrho)
 =  \sup_{\eps>0} \Big( \eps \, \min_{\TT \in \Tc_\eps(\varrho)}(\#\TT - \#\TT_{\rm init})^s \Big).
\end{align}
The minimum in~\eqref{eq2:lemma:apx} exists, since all $\Tc_N$ are finite sets. The minimum in~\eqref{eq3:lemma:apx} exists, since the cardinality is a mapping $\#: \Tnc \to \N$. In either case, the minimizers might not be unique.
If $\Tc = \Tc(\TT_{\rm init})$ is replaced by $\Tnc=~\Tnc(\TT_{\rm init})$, one can define $\A_s^{\rm nc}$, $\Tnc_N$, and $\Tnc_{\varepsilon}(\varrho)$ similarly, and the assertion \eqref{eq3:lemma:apx} holds accordingly.
\end{lemma}

\begin{proof}
We only consider the set $\Tc$ of conforming triangulations, the proof for the set  $\Tnc$ of non-conforming triangulations  follows along the same lines.
For $N \in \N_0$, define $\eps_N := \min_{\TT \in \Tc_N} \varrho(\TT) \ge 0$. 

{\bf Step~1.} 
To prove ``$\ge$'' in~\eqref{eq3:lemma:apx}, let $\eps > 0$.
If $0 < \eps < \eps_0$, there exists a minimal $N \in \N_0$ such that $\min_{\TT \in \Tc_N} \varrho(\TT) \le \eps$. In particular, it follows that $N > 0$, $\Tc_N \cap \Tc_\eps(\varrho) \neq \emptyset$, and $\eps < \min_{\TT \in \Tc_{N-1}} \varrho(\TT)$. This yields that
\begin{align}\label{eq1+:lemma:apx}
 \eps  \, \min_{\TT \in \Tc_\eps(\varrho)}(\#\TT - \#\TT_{\rm init})^s
 \le \min_{\TT \in \Tc_{N-1}} \varrho(\TT) \, N^s 
 \le \sup_{N \in \N_0} \Big( (N+1)^s \min_{\TT \in \Tc_N} \varrho(\TT) \Big) = \A_s^{\rm c}(\varrho).
\end{align}
If $\eps_0 \le \eps$, then $\TT_{\rm init} \in \Tc_{\eps_0}(\varrho) \subseteq \Tc_{\eps}(\varrho)$
and hence the left-hand side of~\eqref{eq1+:lemma:apx} is zero, and~\eqref{eq1+:lemma:apx} thus remains true.
Taking the supremum over all $\eps > 0$, we prove ``$\ge$'' in~\eqref{eq3:lemma:apx}.

{\bf Step~2.}
To prove ``$\le$'' in~\eqref{eq3:lemma:apx}, let $N \in \N_0$.
If $\eps_N > 0$, the definition of $\eps_N$ yields that $\#\TT -\#\TT_{\rm init} \ge N+1$ for all $\TT \in \Tc_{\lambda\eps_N}(\varrho)$ and all $0<\lambda<1$. 
This proves that
\begin{align}\label{eq2+:lemma:apx}
 (N+1)^s \! \min_{\TT \in \Tc_N} \varrho(\TT) 
 \le \min_{\TT \in \Tc_{\lambda\eps_N}(\varrho)} (\#\TT - \#\TT_{\rm init})^s \, \eps_N
 \le \frac{1}{\lambda} \, \sup_{\eps>0}  \big(\eps \! \min_{\TT \in \Tc_\eps(\varrho)}(\#\TT - \#\TT_{\rm init})^s\big).
 \hspace*{-3mm}
\end{align}
If $\eps_N = 0$, then the left-hand side of~\eqref{eq2+:lemma:apx} is zero, and the overall estimate thus remains true. 
Taking the supremum over all $N \in \N_0$, we prove ``$\le$'' in~\eqref{eq3:lemma:apx}
for the limit $\lambda\to1$.
\end{proof}

The following lemma specifies $\varrho(\TT)$ and hence introduces the precise approximation class of the present work.

\begin{lemma}\label{lemma2:apx}
For $s > 0$, let
\begin{align}\label{eq:the rho}
 \A_s^{\rm c} := \A_s^{\rm c}(\varrho),
 \quad \text{where} \quad
 \varrho(\TT): = \eta(\TT; \U_\TT[p_\TT], p_\TT) + \norm{\div \U_\TT[p_\TT]}{\Omega}
 \quad \text{for } \TT \in \Tc.
\end{align}
Then, $\varrho$ satisfies the assumptions of Lemma~\ref{lemma1:apx}. Moreover, there exists a constant 
$C > 0$, which depends only on $\Cstab$ and $\Crel$, such that
\begin{align}
 \label{eq1:lemma2:apx}
 \varrho(\TT) 
 &\le C \, \min_{Q_\TT \in \P(\TT)}
 \big( \eta(\TT; \U_\TT[Q_\TT], Q_\TT) + \norm{\div \U_\TT[Q_\TT]}{\Omega} \big).
\end{align}
\end{lemma}

\begin{proof}
Let $Q_\TT \in \P(\TT)$. 
According to~\eqref{eq:u to p}, we have   that  $\norm{\U_\TT[p_\TT] - \U_\TT[Q_\TT]}{\V}  \le  \norm{p_\TT - Q_\TT}{\P}$. Since $p_\TT$ is the best approximation of $p$ in $\P(\TT)$, it holds that $\norm{p_\TT - Q_\TT}{\P} \le \norm{p - Q_\TT}{\P}$. Hence, 
stability~\eqref{eq:stability} and reliability~\eqref{eq:reliability:stokes} of the error estimator prove that
\begin{align*}
 \varrho(\TT) 
 &\,\,=\,\, \eta(\TT; \U_\TT[p_\TT], p_\TT) + \norm{\div \U_\TT[p_\TT]}{\Omega}
 \\
 &
~\reff{eq:stability}\lesssim \eta(\TT; \U_\TT[Q_\TT], Q_\TT) 
 + \norm{\U_\TT[p_\TT] - \U_\TT[Q_\TT]}{\V} + \norm{p_\TT - Q_\TT}{\P}+ \norm{\div \U_\TT[Q_\TT]}{\Omega}
 \\&
\,\, \lesssim\,\, \eta(\TT; \U_\TT[Q_\TT], Q_\TT) + \norm{\div \U_\TT[Q_\TT]}{\Omega}
 + \norm{p - Q_\TT}{\P}.
 \\&
~\reff{eq:reliability:stokes}\lesssim \eta(\TT; \U_\TT[Q_\TT], Q_\TT) + \norm{\div \U_\TT[Q_\TT]}{\Omega}.
\end{align*}
This proves~\eqref{eq1:lemma2:apx}.
%
With linear convergence (Theorem~\ref{theorem:linearconvergence}), this yields that
\begin{align*}
\inf_{\TT \in \Tc}\varrho(\TT) 
\le \inf_{(i,j,k) \in \QQ} \varrho(\TT_{ijk})
\lesssim \inf_{(i,j,k) \in \QQ} \big( \eta_{ijk} + \norm{\div \U_{ijk}}{\Omega} \big) = 0.
\end{align*}
This concludes the proof.
\end{proof}

Together with Theorem~\ref{theorem:linearconvergence}, the  following theorem is the main result of this work. It states optimal convergence of Algorithm~\ref{algorithm:uzawa}.
The proof is given in Section~\ref{sec:proof of optconv}.

\begin{theorem}\label{thm:optimal}
Let $0<\vartheta<\Cdiv^{-1}$ and $0<\theta<\theta_{\rm opt}:=(1+\Cstab^2 C_{\rm drel}^2)^{-1}$.
Suppose that
\begin{align}\label{eq:kappa}
\kpuno<\theta^{1/2}\Cstab\quad\text{and}\quad\theta<  \sup_{\delta>0} \frac{(1-\kpuno\Cstab)^2\theta_{\rm opt}-(1+\delta^{-1})\kpuno^2\Cstab^2}{1+\delta},
\end{align}
i.e., $0\le\,\kpuno<1$ is sufficiently small.
Moreover, let $0<\kpdue,\kptre<1$ be sufficiently small in the sense of Lemma~\ref{lemma:convergence:j},   Lemma~\ref{lemma:convergence:i}, and  Lemma~\ref{lem:comparison} below.
Then, for all $s>0$, it holds that 
\begin{align}\label{eq:optimal convergence}
 \A_s^{\rm c}<\infty
\quad \Longleftrightarrow\quad
 \sup_{(i,j,k)\in\QQ} \big( \eta_{ijk} + \norm{\div \U_{ijk}}{\Omega} \big)\big(\#\TT_{ijk}-\#\TT_{\rm init}+1\big)^s <\infty.
\end{align}
\end{theorem}

The following remark relates our definition of the approximation class from Lemma~\ref{lemma2:apx} to that of the so-called total error. We refer to Appendix~\ref{appendix:remark} for the proof. 

\begin{remark}\label{dpr:remark-problem}
\rm{(i)} The seminal work~\cite{ks} employs two approximation classes:
\begin{itemize}
\item $\A_s^{\rm c}(\u) := \A_s^{\rm c}(\varrho_{\u})$ for $\varrho_{\u}(\TT) := \min\limits_{\VV_\TT \in \V(\TT)} \norm{\u - \VV_\TT}{\V}$. 
\item $\A_s^{\rm nc}(p) := \A_s^{\rm nc}(\varrho_p)$ for $\varrho_p(\PP) := \min\limits_{Q_\PP \in \P(\PP)} \norm{p - Q_\PP}{\P} = \norm{p - p_\PP}{\P}$.
\end{itemize}
With the data oscillations for any $\PP\in\Tnc$, 
$\osc^2:=\sum_{T\in\PP} \osc_T^2$ where $\osc_T^2:= |T|^{2/n} \, \norm{(1-\Pi_\PP)\f}{T}^2$ for all $T\in\PP$,
we additionally define the approximation class:
\begin{itemize}
\item $\A_s^{\rm nc}(\f) := \A_s^{\rm nc}(\varrho_{\f})$ for $\varrho_{\f}(\PP) :=\osc(\PP)$. 
\end{itemize}

Clearly, the definitions of $\varrho_p$,  $\varrho_{\u}$, and $\varrho_{\f}$ satisfy the assumptions of Lemma~\ref{lemma1:apx}. Moreover,
\begin{align}\label{eq1:dpr:remark-problem}
\A_s^{\rm nc}(p) \simeq \A_s^{\rm c}(p) := \A_s^{\rm c}(\varrho_p)\quad\text{and}\quad\A_s^{\rm nc}(\f) \simeq \A_s^{\rm c}(\f) := \A_s^{\rm c}(\varrho_{\f}).
\end{align}

\rm{(ii)} If we additionally define 
\begin{itemize}
\item $\A_s^{\rm c}(\u,p,\f) := \A_s^{\rm c}(\varrho_{\u,p,\f})$ for $\varrho_{\u,p,\f}(\TT) := \varrho_{\u}(\TT) + \varrho_p(\TT) + \varrho_{\f}(\TT)$,
\end{itemize}
then it holds for all $s>0$ that
\begin{align}\label{eq:apx:remark}
 \frac{1}{3} \, \big(\A_s^{\rm c}(\u) + \A_s^{\rm c}(p) +\A_s^{\rm c}(\f) \big) 
 \le \A_s^{\rm c}(\u,p,\f)
 \le 3^s \, \big(\A_s^{\rm c}(\u) +  \A_s^{\rm c}(p)+\A_s^{\rm c}(\f) \big).
\end{align}
In the literature, cf.\ \cite{ckns,axioms}, the term $\varrho_{\u,p,\f}(\TT)$ is usually referred to as \textit{total error}.

\rm{(iii)}
There hold efficiency and reliability in the sense that
\begin{align}\label{eq3:dpr:remark-problem}
\A_s^{\rm c} \lesssim \A_s^{\rm c}(\u,p,\f) \le \Crel \, \A_s^{\rm c},
\end{align} 
i.e., our approximation class coincides with the one of the total error. 
In particular, if the volume force $\f$ is a $\TT_{\rm init}$-piecewise polynomial of degree less or equal than $m-1$, the oscillations vanish and our approximation class also coincides with that of~\cite[Section~7]{ks}. 

\rm{(iv)}
Note that for smooth $\u,p$, and $\f$ and uniform mesh-refinement, one expects an optimal algebraic convergence rate of $s= m/d$. 
For non-smooth data and adaptive mesh-refinement, the involved approximation classes can be characterized in terms of Besov regularity; see, e.g., \cite{bddp02,gm08,gan17}.
\hfill\qed
\end{remark}

\subsection{Proof of Theorem~\ref{thm:optimal}}\label{sec:proof of optconv}

We start with an auxiliary lemma, which was originally proved in~\cite[Lemma~6.3]{ks}.

\begin{lemma}\label{lemma:refine:pressure}
Let $0<\vartheta<\vartheta'<\Cdiv^{-1}$.
Let $0 < \omega < 1$ be sufficiently small such that 
\begin{align}\label{eq:q2 assumption}
 0 < q := \Cdiv \, \frac{\omega + \vartheta'}{1-\omega} < 1,
\end{align}
 Let $\PP \in \Tnc$ and $\TT\in\Tc(\PP)$. Let $Q_\PP \in \P(\PP)$. 
Let $\VV_\TT \in \V(\TT)$ satisfy that
\begin{align}\label{eq:w assumption}
 \norm{\div(\u[Q_\PP] - \VV_\TT)}{\Omega} \le \omega \, \norm{\div \VV_\TT}{\Omega}.
\end{align}
Then,  $\binev(\PP, \TT,\VV_\TT;\vartheta)$ from Algorithm~\ref{function:refinePressure} 
returns 
$\PP' \in \Tnc(\PP)$ such that the following implication is satisfied for all $\overline\PP \in \Tnc(\PP)$
\begin{align}\label{eq:binev2}
\norm{p - p_{\overline\PP}}{\P}^{2} \le (1-q^2) \, \norm{p - Q_\PP}{\P}^{2}\quad\Longrightarrow\quad \#\PP' - \#\PP \le \Cbin \, ( \#\overline\PP - \#\TT_{\rm init}).
\end{align}
\end{lemma}

\begin{proof}
To see~\eqref{eq:binev2}, 
let $\overline\PP\in\Tnc(\PP)$ with $\norm{p - p_{\overline\PP}}{\P}^{2} \le (1-q^2) \, \norm{p - Q_\PP}{\P}^{2}$.
Note that 
\begin{align}\label{eq1:dpr:010218}
 \norm{p - p_{\widetilde\PP}}{\P}^{2} \le \norm{p - p_{\overline\PP}}{\P}^{2} \le (1-q^2)\, \norm{p - Q_\PP}{\P}^{2},
\quad \text{where} \quad \widetilde\PP:=\PP\oplus\overline\PP\,\in \Tnc(\PP).
\end{align}
The triangle inequality and assumption~\eqref{eq:w assumption} show that
\begin{align*}
\norm{\div \VV_\TT}{\Omega} \le\norm{\div u[Q_\PP]}{\Omega} + \norm{\div(u[Q_\PP]-V_\TT)}{\Omega}~\reff{eq:w assumption}\le  \norm{\div\u[Q_\PP]}{\Omega} + \omega \, \norm{\div \VV_\TT}{\Omega}.
\end{align*}
Hence, Lemma~\ref{lemma:estimator:div} yields that
\begin{align*}
 &q^{2}(1-\omega)^{2}\,\norm{\div \VV_\TT}{\Omega}^{2}
 \le q^{2}\,\norm{\div\u[Q_\PP]}{\Omega}^{2}
 \\& \qquad
~\reff{eq2:estimator:div}\le 
 q^{2}\,\norm{p-Q_\PP}{\P}^{2}
~\reff{eq1:dpr:010218}\le
 \norm{p-Q_\PP}{\P}^2-\norm{p-p_{\widetilde\PP}}{\P}^2
= \norm{p_{\widetilde\PP}-Q_\PP}{\P}^{2}
~\reff{eq:estimator:div}\le  
 \Cdiv^{2} \norm{\Pi_{\widetilde\PP}\nabla\cdot \u[Q_\PP]}{\Omega}^{2}.
\end{align*}
The  triangle inequality together with~\eqref{eq:w assumption} shows that 
\begin{align*}
 \norm{\Pi_{\widetilde\PP}\nabla\cdot \u[Q_\PP]}{\Omega}
\le \norm{\Pi_{\widetilde\PP}\nabla\cdot\VV_\TT}{\Omega}+ \norm{\Pi_{\widetilde\PP}\nabla\cdot( \u[Q_\PP]-\VV_\TT)}{\Omega}\reff{eq:w assumption}\le \norm{\Pi_{\widetilde\PP}\nabla\cdot\VV_\TT}{\Omega}+ \omega\,\norm{\nabla\cdot\VV_\TT}{\Omega}.
\end{align*}
Altogether, we derive that 
\begin{align*}
q(1-\omega)\,\norm{\div \VV_\TT}{\Omega}
\le 
\Cdiv \norm{\Pi_{\widetilde\PP}\nabla\cdot \u[Q_\PP]}{\Omega}
\le 
\Cdiv\big( \norm{\Pi_{\widetilde\PP}\nabla\cdot\VV_\TT}{\Omega}+\omega\,\norm{\div \VV_\TT}{\Omega}\big).
\end{align*}
By choice of $q$ in~\eqref{eq:q2 assumption}, this is equivalent to
\begin{align*}
\vartheta' \, \norm{\div \VV_\TT}{\Omega}
= \frac{q(1-\omega)-\Cdiv \, \omega}{\Cdiv}\,\norm{\div \VV_\TT}{\Omega}\le \norm{\Pi_{\widetilde\PP}\nabla\cdot\VV_\TT}{\Omega}.
\end{align*}
By definition, Algorithm~\ref{function:refinePressure} returns $\PP' \in \Tnc(\PP)$ 
such that
 \begin{align*}\# \PP' - \#\PP \le \Cbin \, (\#\widetilde\PP - \#\PP)
\reff{mesh:overlay}\le \Cbin \, (\#\overline\PP - \#\TT_{\rm init}).
\end{align*}
This concludes the proof.
\end{proof}

The heart of the proof of Theorem~\ref{theorem:linearconvergence} is  the following auxiliary lemma.

\begin{lemma}\label{lem:comparison}
Let $(i,j,k)\in\QQ$ with $k<\kk(i,j)$ and $s>0$.
%
Let $0<\vartheta<\Cdiv^{-1}$ and $0<\theta<\theta_{\rm opt}=(1+\Cstab^2 C_{\rm drel}^2)^{-1}$.
Let $0\le\kpuno<1$ be sufficiently small such that \eqref{eq:kappa} is satisfied.
For sufficiently small $0<\kpdue\ll1$ (see~\eqref{eq:kappa constraints optconv} in the proof below), there exists $C_{\rm comp}$ such that
\begin{align}\label{eq:comparison}
 \#\MM_{ijk} \le C_{\rm comp} ( 1 +(\A_s^{\rm c})^{1/s} )\big( \eta_{ijk} + \norm{\div \U_{ijk}}{\Omega} \big)^{-1/s}.
\end{align}
The constant $C_{\rm comp}>0$ depends only on  
the domain $\Omega$, $\gamma$-shape regularity, the polynomial degree $m$, the parameters $\kpuno,\kpdue,\kptre,\vartheta,\theta$ 
 $\Cmark$, and $s$.
\end{lemma}

\begin{proof}
The proof is split into five steps.

{\bf Step 1.}\quad
Choose
\begin{align}
\varepsilon := \eta_{ijk} + \norm{\div \U_{ijk}}{\Omega}.
\end{align}
Without loss of generality, we may assume that $\eps > 0$
and $\A_s^{\rm c} < \infty$.
Then, Lemma~\ref{lemma1:apx} and Lemma~\ref{lemma2:apx} guarantee the existence of $\overline\TT \in \Tc$ such that
\begin{align}\label{eq:apx}
 \#\overline\TT-\#\TT_{\rm init}
 \le (\A_s^{\rm c}/\varepsilon)^{1/s}
 \quad \text{and} \quad
 \eta(\overline\TT; \U_{\overline\TT}[p_{\overline\TT}], p_{\overline\TT})
 + \norm{\div \U_{\overline\TT}[p_{\overline\TT}]}{\Omega}
 \le \eps.
\end{align}%

{\bf Step 2.}\quad
Define the uniformly refined triangulations
\begin{align*}
 \widehat\TT_0 := \close(\PP_i) \oplus \overline\TT
 \quad \text{and} \quad 
 \widehat\TT_{n+1}:=\refine(\widehat\TT_n,\widehat\TT_n)
 \quad \text{for all $n \in \N_0$.}
\end{align*}
Note that $P_{ij} \in \P(\PP_i) \subseteq \P(\widehat\TT_n)$.
We recall some standard arguments for adaptive mesh-refinement for the (vector-valued) Poisson model problem.
Reliability~\eqref{eq:reliability:velocity}, stability~\eqref{eq:stability}, and reduction~\eqref{eq:reduction} guarantee the existence of  $\Cctr > 0$ and $0 < \qctr < 1$ such that
\begin{align*}
 \eta(\widehat\TT_n; \U_{\widehat\TT_n}[P_{ij}], P_{ij}) 
 \le \Cctr \, \qctr^n \, \eta(\widehat\TT_0; \U_{\widehat\TT_0}[P_{ij}], P_{ij});
\end{align*}
see, e.g.,~\cite[Theorem~4.1 (i)]{axioms}.
According to, e.g.,~\cite[Section~3.4]{axioms}, there exists $\Cmon' > 0$ such that  for all
$ \widehat\TT \in \Tc, \,  
 \widehat\TT' \in~\Tc(\widehat\TT), \,
 P_{\widehat\TT} \in \P(\widehat\TT)$
\begin{align}\label{eq:monotonicity}
 \eta(\widehat\TT'; \U_{\widehat\TT'}[P_{\widehat\TT}], P_{\widehat\TT})
 \le \Cmon' \eta(\widehat\TT; \U_{\widehat\TT}[P_{\widehat\TT}], P_{\widehat\TT})
\end{align}
Note that $\Cctr$, $\qctr$, and $\Cmon'$ depend only on  $\gamma$-shape regularity and the polynomial degree $m$.
With stability~\eqref{eq:stability} and quasi-monotonicity~\eqref{eq:monotonicity}, it follows that
\begin{align*}
 &\eta(\widehat\TT_n; \U_{\widehat\TT_n}[P_{ij}], P_{ij}) 
 \le \Cctr \, \qctr^n \, \eta(\widehat\TT_0; \U_{\widehat\TT_0}[P_{ij}], P_{ij})
 \\&\quad
~\reff{eq:stability}\le \Cctr \, \qctr^n \, \big[ \eta(\widehat\TT_0; \U_{\widehat\TT_0}[p_{\overline\TT}], p_{\overline\TT})
 + \Cstab \big( \norm{\U_{\widehat\TT_0}[P_{ij}] - \U_{\widehat\TT_0}[p_{\overline\TT}]}{\V} + \norm{P_{ij} - p_{\overline\TT}}{\P} \big)
 \big]
 \\&\quad
~\reff{eq:monotonicity}\le \Cctr \, \qctr^n \, \big[ \Cmon' \, \eta(\overline\TT; \U_{\overline\TT}[p_{\overline\TT}], p_{\overline\TT})
 + \Cstab \big( \norm{\U_{\widehat\TT_0}[P_{ij}] - \U_{\widehat\TT_0}[p_{\overline\TT}]}{\V} + \norm{P_{ij} - p_{\overline\TT}}{\P} \big)
 \big].
\end{align*}
With~\eqref{eq:u to p}, we hence  obtain that
\begin{align*}
 \eta(\widehat\TT_n; \U_{\widehat\TT_n}[P_{ij}], P_{ij}) 
 \le \Cctr \, \qctr^n \, 
 \big[ \Cmon' \, \eta(\overline\TT; \U_{\overline\TT}[p_{\overline\TT}], p_{\overline\TT})
 + 2\Cstab \, \norm{P_{ij} - p_{\overline\TT}}{\P} \big].
\end{align*}
According to the reliability estimates \eqref{eq:reliability:stokes} and \eqref{eq2:reliability:stokes}, it holds that
\begin{align*}
 \norm{P_{ij} - p_{\overline\TT}}{\P}
 &\le \norm{p - p_{\overline\TT}}{\P} + \norm{p - P_{ij}}{\P}
 \\
 &\le  \Crel'(\kpuno) \, \big\{ \big( \eta(\overline\TT; \U_{\overline\TT}[p_{\overline\TT}],p_{\overline\TT}) + \norm{\div \U_{\overline\TT}[p_{\overline\TT}]}{\Omega} \big)  + \big( \eta_{ijk} + \norm{\div \U_{ijk}}{\Omega} \big) \big\}.
\end{align*}
By choice of $\overline\TT$ in Step~1 and for $k < \kk(i,j)$, we overall obtain that 
\begin{align}
\begin{split}
\label{eq:hat to ijk}
 \eta(\widehat\TT_n; \U_{\widehat\TT_n}[P_{ij}], P_{ij}) 
 &\,\le \qctr^n \, \Cctr \,
 \big[
 \Cmon' + 4  \,\Cstab\,\Crel'(\kpuno)
 \big] \, \big( \eta_{ijk} + \norm{\div \U_{ijk}}{\Omega} \big)
 \\&\reff{eq1:lemma:convergence:k}\le
 \qctr^n \, \Cctr \, \big[
 \Cmon' +4  \,\Cstab\,\Crel'(\kpuno) 
 \big] \,\frac{1}{\kpdue} \, \Big( 1 + \frac{1}{\kptre} \Big) \, \eta_{ijk}.
 \end{split}
\end{align}%

{\bf Step 3.}\; To shorten notation, we set $\etagals{ijk}[ij]:=\etagal{ijk}[ij]$ and $\gals{ijk}[ij]:=\gal{ijk}[ij].$
Note that  discrete reliability~\eqref{eq:discrete reliability} and stability~\eqref{eq:stability} imply optimality of D\"orfler marking (see, e.g.,~\cite[Section~4.5]{axioms}): 
For any $0 < \theta_\star < \theta_{\rm opt}$,  there exists some $0 < \lambda = \lambda(\theta_\star) \ll 1$ such that, for all $\widecheck\TT \in~\Tc(\TT_{ijk})$, it holds that
\begin{align}\label{eq:doerfler help}
\eta(\widecheck\TT; \U_{\widecheck\TT}[P_{ij}], P_{ij}) 
\le \lambda \, \etagals{ijk}[ij] 
\quad \Longrightarrow \quad
\theta_\star \, (\etagals{ijk}[ij])^2\le \eta(\TT_{ijk}\setminus\widecheck\TT;\U_{ijk},P_{ij}) ^2.
\end{align}
The second inequality in \eqref{eq:doerfler help}, Lemma~\ref{lemma:equivalent est}, and the Young inequality imply for $\delta>0$ that 
\begin{align*}
&(1-\kpuno\Cstab)^2\theta_\star \, \eta_{ijk}^2  \reff{eq:equivalent est}\le \theta_\star(\etagals{ijk}[ij])^2
 \reff{eq:doerfler help}\le \eta(\TT_{ijk}\setminus\widecheck\TT;\gals{ijk}[ij],P_{ij}) ^2
 \\& \qquad
 \reff{eq:stability2}\le (1+\delta) \eta(\TT_{ijk}\setminus\widecheck\TT;\U_{ijk},P_{ij}) ^2+(1+\delta^{-1})\kpuno^2\Cstab^2 \eta_{ijk}^2.
\end{align*}
Due to \eqref{eq:kappa}, we can choose $0 < \theta_\star < \theta_{\rm opt}$ sufficiently close to $\theta_{\rm opt}$ such that
\begin{align}\label{eq:doerfler inexact}
\theta\,\eta_{ijk}^2 \reff{eq:kappa}\le \sup_{\delta>0}\,\frac{(1-\kpuno\Cstab)^2\theta_\star-(1+\delta^{-1})\kpuno^2\Cstab^2}{1+\delta} \, \eta_{ijk}^2\le \eta(\TT_{ijk}\setminus\widecheck\TT;\U_{ijk},P_{ij}) ^2.
\end{align}

Let $\ell \in \N_0$ be the minimal integer such that 
\begin{align*}
\qctr^\ell \, \frac{\Cmon'}{1-\kpuno\Cstab} \Cctr \, \big[
 \Cmon' +4   \Cstab \Crel'(\kpuno)  
 \big] \,\frac{1}{\kpdue} \, \Big( 1 + \frac{1}{\kptre} \Big)
 \le \lambda.
\end{align*}
Recall $\widehat\TT_\ell$ from Step~2.
For $\widecheck\TT := \widehat\TT_\ell \oplus \TT_{ijk}$, it then holds that
\begin{align*}
 \eta(\widecheck\TT; \U_{\widecheck\TT}[P_{ij}], P_{ij})
\reff{eq:monotonicity} \le \Cmon' \, \eta(\widehat\TT_\ell; \U_{\widehat\TT_\ell}[P_{ij}], P_{ij}) 
\reff{eq:hat to ijk}\le \lambda\,(1-\kpuno\Cstab) \, \eta_{ijk}\reff{eq:equivalent est}\le \lambda\etagals{ijk}[ij].
\end{align*}
 Hence, \eqref{eq:doerfler help}--\eqref{eq:doerfler inexact} imply that  $\theta \, \eta_{ijk}^2\le \eta(\TT_{ijk}\setminus\widecheck\TT;\U_{ijk},P_{ij}) ^2$.

{\bf Step 4.}\quad
Since $\MM_{ijk} \subseteq \TT_{ijk}$ in Algorithm~\ref{algorithm:uzawa} {\rm(iv)} has (up to some fixed factor $\Cmark$) minimal cardinality, the overlay estimate \overlay\,  implies that
\begin{align*}
 &\Cmark^{-1} \#\MM_{ijk}
~\reff{eq:doerfler help}\le \#(\TT_{ijk}\setminus\widecheck\TT)
 \le \#\widecheck\TT-\#\TT_{ijk} 
 \stackrel{\overlay}{\le} \#\widehat \TT_\ell-\#\TT_{\rm init}
 \stackrel{\sonestimate}{\le}\Cson^\ell \#\widehat \TT_0
  \\&
 \stackrel{\overlay}{\le}\Cson^\ell \big( \#\close(\PP_i) + \#\overline\TT - \#\TT_{\rm init} \big)
~\reff{eq:apx}\lesssim (\A_s^{\rm c})^{1/s} \, \big( \eta_{ijk} + \norm{\div \U_{ijk}}{\Omega} \big)^{-1/s}
  + \#\close(\PP_i).
\end{align*}%
Elementary calculation (see, e.g.,~\cite[Lemma~22]{bhp17}) shows that
\begin{align*}
 \#\PP - \#\TT_{\rm init} + 1
 \le \#\PP \le 
 \#\TT_{\rm init} \, \big(\#\PP - \#\TT_{\rm init} + 1\big)
 \quad \text{for all } \PP \in \Tnc.
\end{align*}
With $\#\TT_{\rm init} \simeq 1 \lesssim \big( \eta_{ijk} + \norm{\div \U_{ijk}}{\Omega} \big)^{-1/s}$,
the conformity estimate~\eqref{mesh:closure2} yields that
\begin{align*}
 \#\close(\PP_i)
 \lesssim \big( \eta_{ijk} + \norm{\div \U_{ijk}}{\Omega} \big)^{-1/s} + (\#\PP_i - \#\TT_{\rm init}).
\end{align*}
Altogether, this step thus concludes that
\begin{align}\label{eq:aux:Mijk}
 \#\MM_{ijk} 
 \lesssim ( 1 + (\A_s^{\rm c})^{1/s} ) \, \big( \eta_{ijk} + \norm{\div \U_{ijk}}{\Omega} \big)^{-1/s}
 +  (\#\PP_i - \#\TT_{\rm init}).
\end{align}

%

\bigskip

{\bf Step 5.}\quad
Reliability~\eqref{eq2:reliability:velocity} as well as Algorithm~\ref{algorithm:uzawa}  {\rm(ii)}  show for all  $0 \le i' < i$ that
\begin{align*}
 \norm{\nabla \cdot(\u_{i'\jj} -\U_{i'\jj\kk})}{\Omega} 
 \le\norm{\u_{i'\jj} -\U_{i'\jj\kk}}{\V}
 \le\Crel'(\kpuno)\,\eta_{i'\jj\kk}
\le \Crel'(\kpuno)\,\frac{\kpdue}{1-\kpdue} \, \norm{\nabla \cdot \U_{i'\jj\kk}}{\Omega}.
\end{align*}
Let $0<\vartheta<\vartheta'<\Cdiv^{-1}$ and
 $\omega:=\Crel'(\kpuno){\kpdue}/(1-\kpdue)$.
For  $0<\kpdue\ll1$ with
\begin{align} \label{eq:kappa constraints optconv}
0<q := \Cdiv \, \frac{\omega + \vartheta'}{1-\omega}<1, 
\end{align} Lemma~\ref{lemma:refine:pressure} applies and proves  for all  $\overline\PP_{i'} \in \Tnc(\PP_{i'})$ that
\begin{align*}
 \norm{p-p_{\overline\PP_{i'}}}{\P} \le (1-q^2)^{1/2}  \, \norm{p-P_{i'\jj}}{\P}
 \quad\Longrightarrow\quad
  \#\PP_{i'+1}-\#\PP_{i'} 
 \lesssim \#\overline\PP_{i'}-\#\TT_{\rm init}. 
\end{align*}
We choose $\overline\PP_{i'}$ from the definition~\eqref{eq3:lemma:apx} of the approximation norm $\A_s^{\rm c}$ such that 
\begin{align*}
  \#\overline\PP_{i'}-\#\TT_{\rm init} 
 \le (\A_s^{\rm c}/\varepsilon_{i'})^{1/s} 
\quad\text{with}\quad  
\eta(\overline\PP_{i'};\U_{\overline\PP_{i'}}[p_{\overline\PP_{i'}}],p_{\overline\PP_{i'}})+\norm{\div \U_{\overline\PP_{i'}}[p_{\overline\PP_{i'}}]}{\Omega}
 \\
 \le \varepsilon_{i'}:=\frac{(1-q^2)^{1/2}}{\Crel'(\kpuno)} \, \norm{p-P_{i'\jj}}{\P}.
\end{align*}
Reliability~\eqref{eq:reliability:stokes} shows that $ \norm{p-p_{\overline\PP_{i'}}}{\P} \le \Crel\,\big(\eta(\overline\PP_{i'};\U_{\overline\PP_{i'}}[p_{\overline\PP_{i'}}],p_{\overline\PP_{i'}})+\norm{\div \U_{\overline\PP_{i'}}[p_{\overline\PP_{i'}}]}{\Omega}\big)$. 
With $\Crel\le\Crel'(\kpuno)$,  Lemma~\ref{lemma:convergence:i} and
 Lemma~\ref{lem:uniform summability}~(b) yield that 
\begin{align*}
 \#\PP_i - \#\TT_{\rm init}
 = \sum_{i'=0}^{i-1}( \#\PP_{i'+1}-\#\PP_{i'})
 \lesssim (\A_s^{\rm c})^{1/s} \, \sum_{i'=0}^{i-1} \norm{p-P_{i'\jj}}{\P}^{-1/s} 
\stackrel{{\rm(b)}}{\lesssim} (\A_s^{\rm c})^{1/s} \, \norm{p-P_{(i-1)\jj}}{\P}^{-1/s}.
\end{align*}
Next, we prove that $\norm{p-P_{(i-1)\jj}}{\P}^{-1/s} \lesssim \big(\eta_{ijk}+\norm{\div \U_{ijk}}{\Omega}\big)^{-1/s}$. To this end, 
we apply Lemma~\ref{lemma:monotone}~(a)--(d) and Lemma~\ref{lemma:convergence:i}. 
For $i,j>0$,  it holds that
\begin{align*}
 &\eta_{ijk} + \norm{\div\U_{ijk}}{\Omega}
 \stackrel{{\rm(c)}}\lesssim \eta_{ij0} + \norm{\div\U_{ij0}}{\Omega}
 \stackrel{{\rm(b)}}\lesssim \eta_{i(j-1)\kk} + \norm{\div\U_{i(j-1)\kk}}{\Omega}
 \stackrel{{\rm(d)}}\lesssim \eta_{i0\kk} + \norm{\div\U_{i0\kk}}{\Omega}
 \\
 &\quad\stackrel{{\rm(c)}}\lesssim \eta_{i00} + \norm{\div\U_{i00}}{\Omega}
 \stackrel{{\rm(a)}}\lesssim \eta_{(i-1)\jj\kk} + \norm{\div\U_{(i-1)\jj\kk}}{\Omega}
~\reff{eq1:lemma:convergence:i}\simeq \norm{p-P_{(i-1)\jj}}{\P}.
\end{align*}
Note that the overall estimate is also true if $j=0$.  
This proves that $\#\PP_i - \#\TT_{\rm init} \lesssim(\A_s^{\rm c})^{1/s} \,\big(  \eta_{ijk}+\norm{\div \U_{ijk}}{\Omega}\big)^{-1/s}$.  With~\eqref{eq:aux:Mijk}, we obtain that
\begin{align*}
 \#\MM_{ijk} 
 \lesssim ( 1 +(\A_s^{\rm c})^{1/s} ) \, \big(\eta_{ijk}+\norm{\div \U_{ijk}}{\Omega}\big)^{-1/s}.
\end{align*}
This concludes the proof.
\end{proof}

\begin{proof}[Proof of Theorem~\ref{thm:optimal}]
The proof is split into two steps.

\textbf{Step 1.} 
We show the lower bound 
 in~\eqref{eq:optimal convergence}.
Recall that $P_{ij}\in\P(\PP_i)\subseteq\P(\TT_{ijk})$ for all $(i,j,k)\in\QQ$.
Therefore, Lemma~\ref{lemma2:apx} gives that
\begin{align}\label{eq:optimal proof}
\varrho(\TT_{ijk})&\reff{eq1:lemma2:apx}\lesssim\etagal{ijk}[ij]+\norm{\div \gal{ijk}[ij]}{\Omega}
\reff{eq:equivalent est}\simeq \eta_{ijk}+\norm{\div \U_{ijk}}.
\end{align}
If there exists some $(i,j,k)\in\QQ$ such that $\TT_{ijk}=\TT_{i'j'k'}$ for all $(i',j',k')\in\QQ$ with $(i,j,k)\le(i',j',k')$, then, $\varrho(\TT_{i'j'k'})=\varrho(\TT_{ijk})$,~\eqref{eq1:lemma2:apx}, and 
convergence~\eqref{eq:linconv} yield that $\varrho(\TT_{i'j'k'})=0$ and hence $\A_s^{\rm c}<\infty$. 
Otherwise, let $N\in\N_0$ and let $(i,j,k)\in\QQ$ be the largest possible index (with respect to ``$\le$'') such that $\#\TT_{ijk}-\#\TT_{\rm init}\le N$, i.e., $\TT_{ijk}\in\Tc_N$.
Clearly, it holds that $k<\underline k(i,j)$. Therefore,  the son estimate~\eqref{mesh:sons} yields that
\begin{align*}
N+1<\#\TT_{ij(k+1)}-\#\TT_{\rm init}+1  \simeq \#\TT_{ij(k+1)} \reff{mesh:sons}\simeq \#\TT_{ijk}\simeq \#\TT_{ijk}-\#\TT_{\rm init}+1.
\end{align*}
Together with~\eqref{eq:optimal proof}, this leads to
\begin{align*}
\min_{\TT\in\Tc_N} (N+1)^s\rho(\TT)\lesssim (\#\TT_{ijk}-\#\TT_{\rm init}+1)^s\rho(\TT_{ijk}).
\end{align*}
Taking the supremum over all $(i,j,k)\in\QQ$, and then over all $N\in\N_0$, we conclude the first step.

\textbf{Step 2.} 
We show the upper bound 
 in~\eqref{eq:optimal convergence}.
According to the closure estimate~\eqref{mesh:closure} and Lemma~\ref{lem:comparison},  it holds for all $(i',j',k')\in\QQ$ with $\TT_{i'j'k'} \neq \TT_{\rm init}$ that 
\begin{align*}
\#\TT_{i'j'k'}-\#\TT_{\rm init}+1&\,\simeq\,\#\TT_{i'j'k'}-\#\TT_{\rm init}\reff{mesh:closure}\lesssim\sum_{\substack{(i,j,k)\le(i',j',k')\\ k\neq\underline k (i,j)}}\#\MM_{ijk}\\
&\reff{eq:comparison}\lesssim (1+(\A_s^{\rm c})^{1/s})  \sum_{(i,j,k)\le(i',j',k')}\big( \eta_{ijk} + \norm{\div \U_{ijk}}{\Omega} \big)^{-1/s}.
\end{align*}
Hence, linear convergence~\eqref{eq:linconv} in combination with Lemma~\ref{lem:uniform summability}~(a) gives that
\begin{align*}
\#\TT_{i'j'k'}-\#\TT_{\rm init}+1\lesssim(1+(\A_s^{\rm c})^{1/s})^s \big( \eta_{i'j'k'} + \norm{\div \U_{i'j'k'}}{\Omega} \big)^{-1/s} 
\end{align*}
for all for all $(i',j',k')\in\QQ$ with $\TT_{i'j'k'} \neq \TT_{\rm init}$. For all other $(i',j',k')\in\QQ$ with $\TT_{i'j'k'} = \TT_{\rm init}$, the latter estimate is clear. 
With $(1+(\A_s^{\rm c})^{1/s})^s\lesssim 1+\A_s^{\rm c}$, we conclude the proof.
\end{proof}

\appendix
\def\section#1{\refstepcounter{section}
\medskip
\begin{center}
\sc Appendix~\thesection. #1
\end{center}
\medskip
}

\section{Contraction property of $N_\alpha$}
\label{appendix:operator}

\noindent
The norm of a self-adjoint operator $T: H \to H$ on a Hilbert space $H$ satisfies that
\begin{align*}
 \|T\| = \max\{ |\mu|, |M| \},
 \quad \text{where} \quad
 \mu := \inf_{x \in H \backslash \{0\}} \frac{\dual{Tx}{x}_H}{\norm{x}{H}^{2}}
 \text{ and }
 M := \sup_{x \in H \backslash \{0\}} \frac{\dual{Tx}{x}_H}{\norm{x}{H}^{2}}.
\end{align*}
If $T$ is positive semi-definite (i.e., $\dual{Tx}{x}_H \ge 0$ for all $x \in H$), then
\begin{align*}
 \|T\| = \sup_{x \in H \backslash \{0\}} \frac{\dual{Tx}{x}_H}{\norm{x}{H}^{2}}.
\end{align*}
Consider $H = \P$.
Let $0 < \alpha < 2\,\|S\|^{-1}$. Since the Schur complement operator $ S = \nabla \cdot \Delta^{- 1} \nabla: \P \to \P$
is self-adjoint, also the operator $T := I-\alpha S$ is self-adjoint. Moreover, $S$ is positive definite. 
Hence,
\begin{align*}
 \mu = \inf_{q \in \P\, \backslash \{0\}}
 \frac{ \dual{(I-\alpha S)q}{q}_{\Omega} }{ \norm{q}{\Omega}^2 }
 = 1 - \alpha \, \sup_{q \in \P\, \backslash \{0\}} \frac{ \dual{Sq}{q}_{\Omega} }{ \norm{q}{\Omega}^2 }
 = 1 - \alpha \, \|S\| >- 1
\end{align*}
as well as 
\begin{align*}
 M = \sup_{q \in \P\,\backslash \{0\}}
 \frac{ \dual{(I-\alpha S)q}{q}_{\Omega} }{ \norm{q}{\Omega}^2 }
 = 1 - \alpha \, \inf_{q \in \P\, \backslash \{0\}} \frac{ \dual{Sq}{q}_{\Omega} }{ \norm{q}{\Omega}^2 } < 1.
\end{align*}
Altogether, $\|I-\alpha S\| = \max\{ |\mu|, |M| \} < 1$ and thus  $N_\alpha : \P \to \P$ from \eqref{eq:N_alpha} is a contraction.

\section{Proof of~(\ref{eq:appendix:div})}
\label{appendix:div}

\noindent
It suffices to prove the inequality for $\v$ in the dense subspace $C_c^{\infty}(\Omega)^n\subseteq H_0^1(\Omega)=\V$. 
Integration by parts and the fact that $\partial_k\partial_j\v_j = \partial_j\partial_k\v_j$ show that
\begin{align*}
 &\norm{\div \v}{\Omega}^2
 = \sum_{j, k = 1}^n \dual{\partial_j \v_j}{\partial_k \v_k}_\Omega
 = -\sum_{j, k = 1}^n \dual{\partial_k\partial_j \v_j}{\v_k}_\Omega
 = -\sum_{j, k = 1}^n \dual{\partial_j\partial_k \v_j}{\v_k}_\Omega
 \\&\quad
 = \sum_{j, k = 1}^n \dual{\partial_k \v_j}{\partial_j \v_k}_\Omega
 \le \sum_{j, k = 1}^n \norm{\partial_k \v_j}{\Omega} \norm{\partial_j \v_k}{\Omega}
 \le \frac12 \, \sum_{j, k = 1}^n \big( \norm{\partial_k \v_j}{\Omega}^2 + \norm{\partial_j \v_k}{\Omega}^2 \big)
 = \norm{\nabla \v}{\Omega}^2.
\end{align*}

\section{Proof of Remark~\ref{dpr:remark-problem}}
\label{appendix:remark}

{\bf Proof of~(\ref{eq1:dpr:remark-problem}).}
Let $q\in\{p,\f\}$. 
First, $\A_s^{\rm nc}(q) \le \A_s^{\rm c}(q) $ is trivially satisfied due to $\Tc\subseteq\Tnc$.
To see the converse inequality,  let $N\in\N_0$ be arbitrary and $\PP'\in\Tnc_N$ with 
$\varrho_{q}(\PP')=\min_{\PP\in\Tnc_N}\varrho_{q}(\PP)$. 
According to \eqref{mesh:closure2}, we have that $\close(\PP)\in\Tc_{\Cclosure N}$.
Thus, monotonicity of $\varrho_{q}$ gives that 
\begin{align*}
&\min_{\TT\in\Tc_{{\lfloor \Cclosure N \rfloor}}}(\Cclosure N+1)^s\varrho_{q}(\TT)\le (\Cclosure N+1)^s\varrho_{q}(\close(\PP'))
\le (\Cclosure+1)^s (N+1)^s\varrho_{q}(\PP')\\
&\quad=(\Cclosure+1)^s (N+1)^s\min_{\PP\in\Tnc_N}\varrho_{q}(\PP)\le (\Cclosure+1)^s \A_s^{\rm nc}({q}).
\end{align*}
Finally, elementary estimation yields for arbitrary $M\in\N_0$ and $N:=\lfloor M/\Cclosure\rfloor$ that 
\begin{align*}
\min_{\TT\in\Tc_{M}}(M+1)^s\varrho_{q}(\TT)\lesssim\min_{\TT\in\Tc_{\lfloor \Cclosure N \rfloor}}(\Cclosure N+1)^s\varrho_{q}(\TT)\le 2^s \A_s^{\rm nc}({q}).
\end{align*}
Taking the supremum over all $M\in\N_0$, we conclude the proof.\qed

\bigskip

{\bf Proof of~(\ref{eq:apx:remark}).}
By definition, we have that  $\varrho_{\u}(\mesh) + \varrho_p(\mesh)+\varrho_{\f}(\mesh) = \varrho_{\u,p,\f}(\mesh)$. Hence, 
\begin{align*}
\A_s^{\rm c}(\u) + \A_s^{\rm c}(p)+\A_s^{\rm c}(\f)\le 3 \, \A_s^{\rm c}(\u,p,\f).
\end{align*}
Moreover, the overlay estimate~\eqref{mesh:overlay} also proves the converse estimate.

To see this, let $N \in \N_0$. 
If $N \text{ mod } 3=0$, choose $n'=n''=n'''=N/3 \in \N_0$.
If $N \text{ mod }3=2$, choose $n'=(N-1)/3$, $n'' = (N-1)/3 \in \N_0$, $n''' = (N+2)/3\in\N_0$.
If $N \text{ mod }3=1$, choose $n'=(N-2)/3$, $n'' = (N+1)/3 \in \N_0$, $n''' = (N+1)/3\in\N_0$.
Choose $\TT' \in \Tc_{n'}$ such that $\varrho_{\u} (\TT') = \min_{\TT \in \Tc_{n'}}\varrho_{\u} (\TT)$.
Choose $\TT'' \in \Tc_{n''}$ such that $\varrho_{p}(\TT'') = \min_{\TT \in \Tc_{n''}} \varrho_{p}(\TT)$. 
Choose $\TT''' \in \Tc_{n'''}$ such that $\varrho_{p}(\TT''') = \min_{\TT \in \Tc_{n'''}} \varrho_{p}(\TT)$. 
Then, $n' + n'' +n''' = N$ and hence $\TT :=  \TT'\oplus \TT'' \oplus\TT'''  \in \Tc_N$. Moreover, the monotonicity of $\varrho_{\u},\varrho_p$, and $\varrho_{\f}$ yields that
\begin{align*}
 (N+1)^s \, \varrho_{\u,p,\f}(\mesh) 
 &\le \Big(\frac{N+1}{n'+1}\Big)^s \, (n'+1)^s \, \varrho_{\u} (\TT')
 + \Big(\frac{N+1}{n''+1}\Big)^s \, (n''+1)^s \, \varrho_{p}(\TT'')\\
& \quad + \Big(\frac{N+1}{n'''+1}\Big)^s \, (n'''+1)^s \, \varrho_{\f}(\TT''') 
 \le \Big(\frac{N+1}{n'+1}\Big)^s \, \big(  \A_s^{\rm c}(\u) + \A_s^{\rm c}(p)+\A_s^{\rm c}(\f)\big).
\end{align*}
Since $(N+1)/(n'+1) \le 3$, this concludes the proof.\qed

\bigskip

{\bf Proof of~(\ref{eq3:dpr:remark-problem}).}
For all $\TT\in\Tc$, it holds that $(1-\Pi_\TT)(-\nabla p_\TT+\Delta \U_\TT[p_\TT])=0$ and thus $\osc(\TT)\le  \eta(\TT; \U_\TT[p_\TT], p_\TT)$. 
Together with reliability~\eqref{eq:reliability:stokes}, this implies  that $\A_s^{\rm c}(\u,p,\f) \le \Crel \,\A_s^{\rm c}$.
A standard efficiency estimate (see, e.g., \cite[Lemma~4.2]{bmn02}) together with the triangle inequality and \eqref{eq:appendix:div} show that
\begin{align*}
& \eta(\mesh;\U_\mesh[p_\mesh], p_\mesh) + \norm{\div \U_\mesh[p_\mesh]}{\Omega} 
\stackrel{\mbox{\scriptsize\cite{bmn02}}}\lesssim
\norm{\u[p_\mesh]-\U_\mesh[p_\mesh]}{\V}+\osc(\TT)+\norm{\div \u[p_\TT]}{\Omega}\\
&\quad \reff{eq2:estimator:div}\le\norm{\u-\U_\mesh[p_\mesh]}{\V}+\norm{\u-\u[p_\mesh]}{\V}+\norm{p-p_\mesh}{\P} +\osc(\TT)
\\
&\quad\reff{eq:u to p}=\norm{\u-\U_\mesh[p_\mesh]}{\V}+2\norm{p-p_\mesh}{\P}+\osc(\TT).
\end{align*}
The hidden constant depends only on $\TT_{\rm init}$ and the polynomial degree of $m$.
Moreover, it holds that $\U_\mesh: = {\rm argmin}_{\VV_\mesh \in \V(\mesh)} \norm{\u - \VV_\mesh}{\V} = \U_\mesh[p]$. Hence,~\eqref{eq:u to p} shows that  
\begin{align*}
 \norm{\u - \U_\mesh[p_\mesh]}{\V}
 \le \norm{\u - \U_\mesh}{\V} + \norm{\U_\mesh[p] - \U_\mesh[p_\mesh]}{\V}
~\reff{eq:u to p}\le  \norm{\u - \U_\mesh}{\V} +  \norm{p - p_\mesh}{\P}.
\end{align*}
Combining the latter two estimates, we prove for $\TT_{\rm init}$-piecewise polynomial $\f$ that
\begin{align*}
 \eta(\mesh; \U_\mesh[p_\mesh], p_\mesh) + \norm{\div \U_\mesh[p_\mesh]}{\Omega}
 \lesssim \min_{\VV_\mesh \in \V(\mesh)} \norm{\u - \VV_\mesh}{\V} + \min_{Q_\mesh \in \P(\mesh)} \norm{p - Q_{\mesh}}{\P}+\osc(\TT).
\end{align*}
Overall, we thus get the converse estimate $\A_s^{\rm c} \lesssim \A_s^{\rm c}(\u,p,\f)$ and
hence obtain~\eqref{eq3:dpr:remark-problem}.\qed

\section{List of symbols}\label{sec:symbols}
The most important symbols are listed in the following table.

\tiny
\begin{longtable}{p{1.7cm} p{10.5cm} p{2.2cm}}
\hline
Name & Description &First appearance \\
\hline
$a(\cdot,\cdot)$ & bilinear form corresponding to $-\Delta$ & Section~\ref{subsec:contStokesprob}\\
$A$ & operator corresponding to $-\Delta$ & Section~\ref{subsec:contStokesprob}\\
$\A_s^{\rm c}$ & approximation constant on conforming triangulations & Lemma~\ref{lemma2:apx}\\
$\A_s^{\rm c}(\cdot)$ & approximation constant for given quantity on conforming triangulations & Lemma~\ref{lemma1:apx}\\
$\A_s^{\rm nc}$ & approximation constant on non-conforming triangulations & Lemma~\ref{lemma2:apx}\\
$\A_s^{\rm nc}(\cdot)$ & approximation constant for given quantity on non-conforming triangulations & Lemma~\ref{lemma1:apx}\\
$b(\cdot,\cdot)$ & bilinear from corresponding to $-\nabla\cdot$ & Section~\ref{subsec:contStokesprob}\\
$B$ & operator corresponding to $-\nabla\cdot$ & Section~\ref{subsec:contStokesprob}\\
$B'$ & operator corresponding to $\nabla$ & Section~\ref{subsec:contStokesprob}\\
$\binev(\cdot,\cdot,\cdot;\cdot)$ & output of Binev algorithm & Algorithm~\ref{function:refinePressure}\\
$\bisect(\cdot,\cdot)$ & non-conforming refinement function & Section~\ref{subsec:partitions}\\
$C_1$ & linear convergence constant in $k$-direction   & Lemma~\ref{lemma:convergence:k}\\
$C_2$ & linear convergence constant in $j$-direction & Lemma~\ref{lemma:convergence:j}\\
$C_3$ & linear convergence constant in $i$-direction & Lemma~\ref{lemma:convergence:i}\\
$\Cbin$ & Binev constant & Section~\ref{subsec:binev}\\
$\Cclosure$ & constant in closure estimate & Section~\ref{subsec:partitions}\\
$C_{\rm comp}$ & comparison constant & Lemma~\ref{lem:comparison}\\
$\Cdiv$ & equivalence constant for norms on pressure space& Section~\ref{subsec:contStokesprob}\\
$\Cdrel$ &discrete reliability constant & Lemma~\ref{lemma:discrete_reliability}\\
$\Clin$ & linear convergence constant & Theorem~\ref{theorem:linearconvergence}\\
$\Cmark$  & marking constant of adaptive algorithm & Algorithm~\ref{algorithm:uzawa}\\
$\Cmon$ & monotonicity constant for estimator & Lemma~\ref{lemma:monotone}\\
$\Cred$ & reduction constant& Lemma~\ref{lemma:reduction}\\
$\Crel$ &reliability constant & Lemma~\ref{lem:reliability}\\
$\Crel'(\cdot)$ & reliability constant for adaptive algorithm & Lemma~\ref{lemma:equivalent est}\\
$\Cson$ & maximal number of sons & Section~\ref{subsec:partitions}\\
$\Cstab$ &stability constant for estimator & Lemma~\ref{lemma:stability}\\
$\close(\cdot)$ & conforming closure of triangulation & Section~\ref{subsec:partitions}\\
$d$ & dimension & Section~\ref{subsec:model problem}\\
$\eta$ & error estimator & Section~\ref{section:estimation}\\
$\eta_{ijk}$ & error estimator of adaptive algorithm & Section~\ref{section:adaptive algorithm}\\
$\eta_T$ & error indicator on an element & Section~\ref{section:estimation}\\
$\f$ & given body force & Section~\ref{subsec:model problem}\\
$\gamma$ & shape regularity constant & Section~\ref{subsec:partitions}\\
$\underline j$  & maximal index $j$ for given index $i$ & Lemma~\ref{lemma:set:Q}\\
$\underline k$  & maximal index $k$ for given indices $(i,j)$ & Lemma~\ref{lemma:set:Q}\\
$\kappa_1$ & parameter of adaptive algorithm to approximate Galerkin approximation & Algorithm~\ref{algorithm:uzawa}\\
$\kappa_2$ & parameter for $i$ direction of adaptive algorithm & Algorithm~\ref{algorithm:uzawa}\\
$\kappa_3$ &parameter for $j$ direction of adaptive algorithm & Algorithm~\ref{algorithm:uzawa}\\
$m$ & polynomial degree & Section~\ref{subsec:discrete}\\
$\Omega$ &bounded Lipschitz domain & Section~\ref{subsec:model problem}\\
$p$ &exact pressure & Section~\ref{subsec:model problem}\\
$p_i$ & best approximation in discrete pressure space of adaptive algorithm & Section~\ref{section:adaptive algorithm}\\
$p_\PP$ & best approximation in discrete pressure space  & Section~\ref{section:auxiliary}\\
$P_{ij}$ & approximative pressure of adaptive algorithm & Section~\ref{section:adaptive algorithm}\\
$\P$ &pressure space & Section~\ref{subsec:model problem}\\
$\P(\cdot)$ & discrete pressure space on non-conforming triangulation & Section~\ref{subsec:discrete}\\
$\P_i$ & discrete pressure space of adaptive algorithm & Section~\ref{section:adaptive algorithm}\\
$\PP_i$ & non-conforming triangulation for pressure of adaptive algorithm & Section~\ref{section:adaptive algorithm}\\
$\Pi_i$ & $L^2$-orthogonal projection on non-conforming triangulation of adaptive algorithm & Section~\ref{section:adaptive algorithm}\\
$\Pi_\PP$ & $L^2$-orthogonal projection on non-conforming triangulation & Section~\ref{section:auxiliary}\\
$q_1$ & linear convergence constant in $k$-direction between $0$ and $1$ & Lemma~\ref{lemma:convergence:k}\\
$q_2$ & linear convergence constant in $j$-direction between $0$ and $1$ & Lemma~\ref{lemma:convergence:j}\\
$q_3$ & linear convergence constant in $i$-direction between $0$ and $1$ & Lemma~\ref{lemma:convergence:i}\\
$\qlin$ &linear convergence constant between $0$ and $1$ & Theorem~\ref{theorem:linearconvergence}\\
$\qred$ & reduction constant between $0$ and $1$ & Lemma~\ref{lemma:reduction}\\
$\QQ$  &set of possible indices & Lemma~\ref{lemma:set:Q}\\
$\refine(\cdot,\cdot)$ & conforming refinement function & Section~\ref{subsec:partitions}\\
$S$ & Schur complement operator & Section~\ref{subsec:contStokesprob}\\
$\Tc$ & set of conforming triangulations & Section~\ref{subsec:partitions}\\
$\Tc(\cdot)$ & set of conforming refinements & Section~\ref{subsec:partitions}\\
$\mathbb{T}_\varepsilon^{\rm c}(\cdot)$ & set of conforming triangulations with given quantity below $\varepsilon$ & Lemma~\ref{lemma1:apx}\\
$\mathbb{T}_N^{\rm c}$ & set of conforming triangulations with bounded element number & Lemma~\ref{lemma1:apx}\\
$\Tnc$ & set of non-conforming triangulations & Section~\ref{subsec:partitions}\\
$\Tnc(\cdot)$ & set of non-conforming refinements & Section~\ref{subsec:partitions}\\
$\mathbb{T}_\varepsilon^{\rm nc}(\cdot)$ & set of non-conforming triangulations with given quantity below $\varepsilon$  & Lemma~\ref{lemma1:apx}\\
$\mathbb{T}_N^{\rm nc}$ & set of non-conforming triangulations with bounded element number & Lemma~\ref{lemma1:apx}\\
$\TT_{ijk}$ & conforming triangulation for velocity of adaptive algorithm  & Section~\ref{section:adaptive algorithm}\\
$\TT_{\rm init}$ & initial conforming triangulation & Section~\ref{subsec:partitions}\\
$\vartheta$ & parameter of Binev algorithm & Algorithm~\ref{function:refinePressure}\\
$\theta$ & D\"orfler marking parameter of adaptive algorithm & Algorithm~\ref{algorithm:uzawa}\\
$\theta_{\rm opt}$ & threshold for  D\"orfler marking parameter & Algorithm~\ref{algorithm:uzawa}\\
$\u$ &exact velocity & Section~\ref{subsec:model problem}\\
$\u[\cdot]$ & exact velocity for given pressure & Section~\ref{section:auxiliary}\\
$\u_{ij}$ & exact velocity to approximate pressure of adaptive algorithm & Section~\ref{section:adaptive algorithm}\\
$\u_\PP$ & exact velocity for best approximation in discrete pressure space & Section~\ref{section:auxiliary}\\
$\U_{ijk}$ & approximative velocity of adaptive algorithm & Section~\ref{section:adaptive algorithm}\\
$\U_\TT[\cdot]$& Galerkin approximation of velocity for given pressure & Section~\ref{section:auxiliary}\\
$\V$ &velocity space & Section~\ref{subsec:model problem}\\
$\V(\cdot)$ &discrete velocity space on conforming triangulation & Section~\ref{subsec:discrete}\\
$\V_{ijk}$ & discrete velocity space of adaptive algorithm & Section~\ref{section:adaptive algorithm}\\
$\dual{\cdot}{\cdot}_\Omega$ & $L^2$-scalar product & Section~\ref{subsec:contStokesprob}\\
$\norm{\cdot}{\Omega}$ & $L^2$-norm & Section~\ref{subsec:contStokesprob}\\
$\norm{\cdot}{\P}$ & norm on pressure space & Section~\ref{subsec:contStokesprob}\\
$\norm{\cdot}{\V}$ & norm on velocity space & Section~\ref{subsec:contStokesprob}\\
$|(\cdot,\cdot,\cdot)|$ & number of iterations to reach given indices & Section~\ref{sec:convergence thm}\\
$\oplus$ & overlay of two triangulations & Section~\ref{subsec:partitions}\\
$<$ & order relation on set of possible indices & Section~\ref{sec:convergence thm}\\
\hline
\end{longtable}

\def\section#1#2{\refstepcounter{section}
\medskip
\begin{center}
\sc #2
\end{center}
\medskip
}

\bibliographystyle{alpha}
\bibliography{literature}

\end{document}